\documentclass[reqno]{amsart}
\usepackage{amssymb}
\usepackage{xcolor}
\usepackage{graphicx}

\usepackage{mathtools}
\mathtoolsset{showonlyrefs}

\newtheorem{Lemma}{Lemma}[section]
\newtheorem{Theorem}[Lemma]{Theorem}
\newtheorem{Proposition}[Lemma]{Proposition}
\newtheorem{Corollary}[Lemma]{Corollary}
\newtheorem{Remark}[Lemma]{Remark}
\newtheorem{Definition}[Lemma]{Definition}
\newtheorem{Notation}[Lemma]{Notation}
\newtheorem{Problem}[Lemma]{Inverse problem}

\newcommand{\removedEinstein}[1]{}

\newcommand{\extension}[1]{} 
\newcommand{\generalizations}[1]{}

\newcommand{\hiddenfootnote}[1]{}

\DeclareMathOperator{\Span}{span}
\DeclareMathOperator{\dist}{dist}
\renewcommand{\L}{{\mathcal L}} 
\newcommand{\p}{\partial}
\newcommand{\N}{{\mathbb N}}
\newcommand{\R}{{\mathbb R}}

\title[Boundary distance map on Finsler manifolds]{Determination of a compact Finsler manifold from its boundary distance map and an inverse problem in elasticity}

\date{\today}

\subjclass[2010]{86A22, 53Z05, 53C60}
\keywords{inverse problems, Finsler geometry, distance function, elastic waves}

\begin{document}
\author[M. V. de Hoop]{Maarten V.  de Hoop $^{\diamond}$}

\author[J. Ilmavirta]{Joonas Ilmavirta $^\dagger$}

\author[M. Lassas]{Matti Lassas $^\square$}

\author[T. Saksala]{Teemu Saksala $^{\diamond, \ddagger}$}



{
\let\thefootnote\relax\footnote{$^\diamond$ Department of Computational and Applied Mathematics, Rice University, USA}
\let\thefootnote\relax\footnote{$^\dagger$ Unit of Computing Sciences, Tampere University,  Finland}
\let\thefootnote\relax\footnote{$^\square$ Department of Mathematics and Statistics, University of Helsinki, Finland}
\let\thefootnote\relax\footnote{$^\ddagger$ Department of Mathematics, North Carolina State University, USA}
}


\email[de Hoop]{mdehoop@rice.edu}
\email[Ilmavirta]{joonas.ilmavirta@tuni.fi}
\email[Lassas]{matti.lassas@helsinki.fi}
\email[Saksala]{tssaksal@ncsu.edu}


\bigskip

\begin{abstract}
We prove that the boundary distance map of a smooth compact Finsler manifold with smooth boundary determines its topological and differentiable structures. We construct the optimal fiberwise open subset of its tangent bundle and show that the boundary distance map determines the Finsler function in this set but not in its exterior.  If the Finsler function is fiberwise real analytic, it is determined uniquely. We also discuss the smoothness of the distance function between interior and boundary points.

We recall how the fastest $qP$-polarized waves in anisotropic elastic medium are a given as solutions of the second order hyperbolic pseudodifferential equation $(\frac{\p^2}{\p t^2}-\lambda^1(x,D))u(t,x)=h(t,x)$ on $\R^{1+3}$, where $\sqrt{\lambda^1}$ is the Legendre transform of a fiberwise real analytic Finsler function $F$ on $\R^3$. If $M \subset \R^3$ is a $F$-convex smooth bounded domain we say that a travel time of $u$ to $z \in \p M$ is the first time $t>0$ when the wavefront set of $u$ arrives in $(t,z)$. The aforementioned geometric result can then be utilized to determine the isometry class of $(\overline M,F)$ if we have measured a large amount of travel times of $qP$-polarized waves, issued from a dense set of unknown interior point sources on $M$. 
\end{abstract}

\maketitle


\section{Introduction}
\label{Se:inverse_problem}

This paper is devoted to an inverse problem for smooth compact Finsler
manifolds with smooth boundaries. We prove that the boundary distance
map of such a manifold determines its topological and differentiable
structures. In general, the boundary distance map is not sufficient to
determine the Finsler function in those directions which correspond to
geodesics that are either trapped or are not distance minimizers to
terminal boundary points. To prove our result, we embed a Finsler
manifold with boundary into a function space and use smooth boundary
distance functions to give a coordinate structure and the Finsler
function where possible.

This geometric problem arises from the propagation of singularities
from a point source for the elastic wave equation. The point source can be natural (e.g. an earthquake as a source of seismic waves) or
artificial (e.g. produced by focusing of waves or by a wave
sent in scattering from a point scatterer). Due to polarization effects, there are
singularities propagating at various speeds. We study the first
arrivals and thus restrict our attention to the fastest singularities
(corresponding to so-called \textit{qP}-polarization, informally
``primary waves''). They follow the geodesic flow of a Finsler
manifold, as we shall explain in more detail in
section~\ref{Se:elastic}. 
The goal of this paper is to study a geophysical inverse problem in the framework of Finsler geometry and then bring the conclusion back to the elastic model. 

An elastic body --- e.g. a planet --- can be modeled as a manifold,
where distance is measured in travel time: The distance between two
points is the shortest time it takes for a wave to go from one point
to the other. If the material is elliptically anisotropic, then this
elastic geometry is Riemannian. However, this sets a very stringent
assumption on the stiffness tensor describing the elastic system, and
Riemannian geometry is therefore insufficient to describe the
propagation of seismic waves in the Earth. We make no structural
assumptions on the stiffness tensor apart from the physically
necessary symmetry and positivity properties, and this leads
necessarily to Finsler geometry.

The inverse problem introduced above can be rephrased as the following
problem in geophysics.
Imagine that earthquakes occur at known times
but unknown locations within Earth's interior and arrival times are
measured everywhere on the surface.
Are such travel time measurements
sufficient to determine the possibly anisotropic elastic wave speed
everywhere in the interior and pinpoint the locations of the
earthquakes?
While earthquake times are not  known in practice, this
is a fundamental mathematical problem that underlies more elaborate
geophysical scenarios.
In the Riemannian realm the corresponding
result \cite{Katchalov2001, kurylev1997multidimensional} is a crucial
stepping stone towards the results of
\cite{belishev1992reconstruction, deHoop1, de2018inverse,
ivanov2018distance, kurylev2010rigidity, LaSa}. We expect that
solutions to inverse problems for the fully anisotropic elastic wave
equation rely on geometrical results similar to the ones presented in
this paper.
However, as Finsler geometry is substantially more complicated than Riemannian --- especially concerning the structure of unit spheres --- the present result is not a mere straightfoward generalization of the corresponding Riemannian result nor do we expect follow-up results to be so.

\subsection{Main results}

We let $(M,F)$ be a smooth compact, connected Finsler manifold with
smooth boundary $\p M$ (For the basic of theory of Finsler manifolds see the appendix (Section~\ref{Se:Appendix1}) at the end of this paper.). We denote the tangent bundle of $M$ by $TM$ and use the notation $(x,y)$ for points in $TM$, where $x$ is a base point and $y$ is a vector in the fiber $T_xM$. The notations $T^\ast M$ and $(x,p)$ are reserved for the cotangent bundle and its points respectively. 

\medskip

We write $d_F:M\times M \to \R$ for the non-symmetric distance function given by a Finsler function $F$.
For a given $x\in M$ the boundary distance function related to $x$ is
\begin{equation}
r_x\colon\p M \to [0,\infty), \quad r_{x}(z)=d_F(x,z).
\end{equation}
To give the mapping $x\mapsto r_x$ a name, we denote $r_x=\mathcal R(x)$.
We denote by $\mathcal{R}(M^{int})=\{r_x:x\in M^{int}\}$ the collection
of all boundary distance functions with interior source points.
inverse problem with \textit{boundary distance data}
\begin{equation}
\label{eq:data}
(\mathcal{R}(M^{int}), \p M).
\end{equation}
\begin{Problem}
Do the boundary distance data~\eqref{eq:data} determine $(M,F)$ up to isometry?
\end{Problem}

We emphasize that we do not assume $d_F$ to be
symmetric and therefore data~\eqref{eq:data} contain only
information where the distance is measured from the points of
$M^{int}$ to points in $\p M$. We note that for any $x \in M^{int}$,
\begin{equation}
r_x(z)=d_{\stackrel{\leftarrow}{F}}(z,x), \quad z \in \p M,
\end{equation}
where $\stackrel{\leftarrow}{F}$ is the Finsler function 
\begin{equation}
\label{eq:reversed_Fins}
\stackrel{\leftarrow}{F}(x,y):=F(x,-y). 
\end{equation}
Therefore, data~\eqref{eq:data} are equivalent to the data
\begin{equation}
(\{d_{\stackrel{\leftarrow}{F}}(\cdot,x)\colon \p M \to \R \: |\: x \in M^{int}\},\p M),
\end{equation}
where the distance is measured from the boundary to the interior. In \cite{Katchalov2001, kurylev1997multidimensional} it is shown that the data~\eqref{eq:data} determine a Riemannian manifold $(M,g)$ up to isometry. In the Finsler case this is not generally true. Next we explain what can be obtained from the Finslerian boundary distance data~\eqref{eq:data}.

\color{black}
\begin{Notation}
\label{De:good_set}
Let $(M,F)$ be a Finsler manifold with boundary.
For any $(x,y)\in TM\setminus 0$, let $\gamma_{x,y}$ denote the geodesic starting at $x$ in the direction of $y$.
We denote by $t(x,y)\in[0,\infty]$ be the first time the geodesic $\gamma_{x,y}$ meets the boundary, and we denote this boundary point by $z(x,y)=\gamma_{x,y}(t(x,y))$.
We denote by $G(M,F)$ the set of points $(x,y)\in TM\setminus0, \: x \in M^{int}$ for which $t(x,y)<\infty$ and the geodesic $\gamma_{x,y}$ is minimizing between the initial point $x$ and the final point $z(x,y)$ on the boundary.
\end{Notation}

Under this convention we have that $t(x,y)=0$ when $x\in\partial M$ and $t(x,y)>0$ whenever $x\in M^{int}$.
Our definition of the exit time $t(x,y)$ guarantees that $\gamma_{x,y}(0,t(x,y))\subset M^{int}$.

Since for any $(x,y) \in TM\setminus \{0\}$ and $a >0$ it holds that $\gamma_{x,ay}(t)=\gamma_{x,y}(at)$, we notice that $G(M,F)$ is a conic set. Let $(x,y) \in G(M,F)$, then $t(x,ay)=a^{-1}t(x,y)$ and $z(x,y)=z(x,ay)$ for any $a>0$. Moreover if $F(y)=1$, then $t(x,y)=d_F(x,z(x,y))$. 

We show that the data~\eqref{eq:data} determine the Finsler function in the closure of the set $G(M,F)$ and that the data~\eqref{eq:data} are not sufficient to recover the Finsler function $F$ on $TM^{int} \setminus \overline{G(M,F)}$. The reason is that  the data~\eqref{eq:data} do not provide any information about the geodesics that are trapped in $M^{int}$ or do not minimize the distance between the point of origin and the terminal boundary point. Therefore, to recover the Finsler function $F$ globally we assume that for every $x \in M$ the function $F(x, \cdot)\colon T_x M \setminus\{0\} \to \R$ is real analytic. We call such a Finsler function \textit{fiberwise real analytic}. For instance Finsler functions $F(x,y)=\sqrt{g_x(y,y)}$, where $g$ is a Riemannian metric, and Randers metrics are fiberwise real analytic. In Section~\ref{Se:elastic} we show that also the Finsler metric related to the fastest polarization of elastic waves is fiberwise real analytic.

\medskip
Now we formulate our main theorems. If $(M_i,F_i), \: i \in \{1,2\}$ are smooth, connected, compact Finsler manifolds with smooth boundaries, we call a smooth map 
\\
$\Phi\colon(M_1,F_1)\to(M_2,F_2)$ a \textit{Finslerian isomorphism} if it is a diffeomorphism which satisfies
\begin{equation}
F_1(x,y)=F_2(\Phi(x),\Phi_\ast y), \quad (x,y) \in TM_1.
\end{equation}
Here $\Phi_\ast$ is the pushforward by $\Phi$. We say that the boundary distance data of manifolds $(M_i,F_i), \: i=1,2$ agree, if there exists a diffeomorphism $\phi\colon \p M_1\to \p M_2$ such that 
\begin{equation}
\label{eq:BDD_agree_start}
\{r_{x_1}:x_1\in M_1^{int}\}=\{r_{x_2}\circ \phi:x_2\in M_2^{int}\}\subset C(\p M_1).
\end{equation}
We emphasize that this is an equality of
non-indexed sets and we do not know the point $x_1$ corresponding to
the function $r_{x_1}$. 

Our first main result shows that the boundary distance data~\eqref{eq:data} determine a manifold upto  a diffeomorphism and a Finsler function in an optimal set.


\begin{Theorem}
\label{Th:smooth}
Let $(M_i,F_i), \: i=1,2$ be smooth, connected, compact Finsler
manifolds with smooth boundaries.
Suppose that there is a
diffeomorphism 
$\phi\colon\p M_1\to \p M_2$ so that
\eqref{eq:BDD_agree_start} holds. Then there is a diffeomorphism
$\Psi\colon M_1\to M_2$ so that $\Psi|_{\p M_1}=\phi$.
The sets $\overline{ G(M_1,F_1})$ and $\overline{ G(M_1,\Psi^*F_2)}$ coincide
and in this set $F_1=\Psi^*F_2$, where the pullback $\Psi^*F_2:TM_1\to \R$ is
defined by
\begin{equation}
\Psi^*F_2(x,y)=F_2(\Psi(x),\Psi_\ast y).
\end{equation}
Moreover, for any $(x_0,y_0)\in TM_1^{int}\setminus \overline{G(M_1,F_1)}$ there
exists a smooth Finsler function $F_3\colon TM_1\to[0,\infty)$ so that
$d_{F_1}(x,z)=d_{F_3}(x,z)$ for all $x\in M_1$ and $z\in\partial M_1$ but
$F_1(x_0,y_0)\neq F_3(x_0,y_0)$.
\end{Theorem}

\begin{Remark} The set $G(M,F)$ can be large or small as the following examples illustrate. If every geodesic of $(M,F)$ is minimizing, then it holds that $\overline{
G(M,F)}=TM$. This holds for instance on simple Finsler manifolds, as by the definition for any pair of points there exists a unique distance minimizing geodesic that depends smoothly on these points. If~$M$ is any subset of $S^2$ larger than the hemisphere and if~$F$ is given by the round metric, then $TM^{int} \setminus \overline{G(M,F)}$ contains an open non-empty set~$U$ whose canonical projection to $M$ is an open neighborhood of the equator.
\end{Remark}



Our second main result shows that the boundary distance data~\eqref{eq:data} determine a fiberwise real analytic Finsler manifold up to an isometry. 

\begin{Theorem}
\label{Th:analytic}
Let $(M_i,F_i), \: i=1,2$ be smooth, connected, compact Finsler
manifolds with smooth boundary.
Suppose that there is a
diffeomorphism 
$\phi\colon\p M_1\to \p M_2$ such that
\eqref{eq:BDD_agree_start} holds.
If Finsler functions $F_i$ are fiberwise
real analytic, then there exists a Finslerian isometry $\Psi \colon
(M_1,F_1)\to (M_2,F_2)$ so that $\Psi|_{\p M_1}=\phi$.
\end{Theorem}

In Section~\ref{Se:elastic} we consider anisotropic elastic wave equation on $\R^{1+3}$. We show that under physically necessary symmetry and positivity properties on the stiffness tensor the travel times of the fastest polarized $qP$-waves are given by a distance function of a fiberwise real analytic Finsler function $F$. Therefore Theorem~\ref{Th:analytic} can be used to recover the isometry class $(\overline M,F)$ from the travel times of these waves. Here $M \subset \R^3$ is an open bounded subset of $\R^3$ that has a smooth boundary. See Theorem~\ref{Th:elastic} and the discussion preceding it for the rigorous claims.   

\begin{Remark}
In Theorems~\ref{Th:smooth} and~\ref{Th:analytic} we measure distances
from the interior to the boundary. If we measure in the opposite
direction, from boundary to the interior, this corresponds to the data
\eqref{eq:data} given with respect to the reversed Finsler function
$\stackrel{\leftarrow}{F}(x,y)$. Our results give uniqueness for
$\stackrel{\leftarrow}{F}$ and therefore $F$. That is, our main
results hold no matter which way distances are measured.
\end{Remark}


\subsubsection{Outline of the proofs of the main results}

Theorem~\ref{Th:analytic} essentially follows from Theorem
\ref{Th:smooth}. We split the proof of Theorem~\ref{Th:smooth} into
four parts (subsections 3.1--3.4). In the first part, we show that the
data~\eqref{eq:data} determine $r_x$ for any $x \in M$. Then we study
the properties of the map $ \mathcal{R}:M \ni x \mapsto r_x \in
L^\infty(\p M)$ and show that this map is a topological embedding. We
use the map $\mathcal{R}$ to construct a map $\Psi: (M_1,F_1) \to
(M_2,F_2)$ that will be shown to be a homeomorphism. In the second
part, we show that the map $\Psi$ is a diffeomorphism.  In the third part we connect the set $G(M,F)$ to smoothness of the distance functions of the form $d_F(\cdot,z), \: z \in \p M$. In the final section we use this to prove that the map $\Psi$ is a Finslerian
isometry in the optimal set $G(M,F)$. 

We have included in this paper a supplemental Section~\ref{Se:Appendix2} and the appendix  (Section~\ref{Se:Appendix1}), which contain necessary material for the
proof of Theorem~\ref{Th:smooth}. We have also included some well-known results and
properties in the Riemannian case while providing a detailed
background of compact Finsler manifolds with and without boundary for
the proof given in Section~\ref{Se:proof}. To the best of our
knowledge, most of this material cannot be found in the literature.

\subsection{Background and related work}

\subsubsection{Geometric inverse problems}

The problem of determination of the isometry type of a Riemannian manifold from its boundary distance map has been studied \cite{Katchalov2001,  kurylev1997multidimensional}. The construction of the topology of the manifold was introduced in~\cite{kurylev1997multidimensional} and the reconstruction of the smooth atlas on the manifold and the metric tensor in these coordinates was considered in~\cite{Katchalov2001}. We emphasize that there results were heavily based on Riemannian geometry.
The problem of determining a  Riemannian manifold  from its boundary distance map
is related to many other
geometric inverse problems. For instance, it is a crucial step in
proving uniqueness for Gel'fand's inverse boundary spectral problem
\cite{Katchalov2001}. Gel'fand's problem concerns the question whether
the data
\begin{equation}
\{\p M, (\lambda_j, \p_\nu \phi_j|_{\p M})_{j=1}^\infty\}
\end{equation}
determine $(M,g)$ up to isometry. Above $(\lambda_j, \phi_j)$ are the Dirichlet eigenvalues and the corresponding $L^2$-orthonormal eigenfunctions of the Laplace-Beltrami operator.  Belishev and Kurylev provide an
affirmative answer to this problem in
\cite{belishev1992reconstruction}.

We recall that the Riemannian wave operator is a globally hyperbolic
linear partial differential operator of real principal
type. Therefore, the Riemannian distance function and the propagation
of a singularity initiated by a point source in space time are related
to one another. In other words, $r_x(z)=t(z)-s$, where $t(z)$ is the
time when the singularity initiated by the point source $(s,x) \in
(0,\infty )\times M$ hits $z\in \p M$. If the initial time $s$ is
unknown, but the arrival times $t(z), z\in \p M$ are known, then one
obtains a boundary distance difference function
$D_x(z_1,z_2):=r_x(z_1)-r_x(z_2), \: z_1,z_2 \in \p M$. In~\cite{LaSa}
it is shown that if $U\subset N$ is a compact subset of a closed
Riemannian manifold $(N,g)$ and $U^{int}\neq \emptyset$, then
\textit{distance difference data} $\{(U,g|_{U}), \{D_x\colon U\times U
\to \R \:| \:x \in N\}\}$ determine $(N,g)$ up to isometry. This
result was recently generalized to complete Riemannian manifolds
\cite{ivanov2018distance} and for compact Riemannian manifolds with
boundary~\cite{de2018inverse}.

If the sign in the definition of the distance difference functions is
changed, we arrive at the distance sum functions,
\begin{equation}
\label{eq:dist sum}
D^+_x(z_1,z_2)=d(z_1,x)+d(z_2,x),\quad x\in M,\ z_1,z_2\in \p M .
\end{equation}
These functions give the lengths of the broken geodesics, that is, the
union of the shortest geodesics connecting $z_1$ to $x$ and the
shortest geodesics connecting $x$ to $z_2$. Also, the gradients of
$D^+_x(z_1,z_2)$ with respect to $z_1$ and $z_2$ give the velocity
vectors of these geodesics. The inverse problem of determining the
manifold $(M,g)$ from the \textit{broken geodesic data}, consisting of
the initial and the final points and directions, and the total length,
of the broken geodesics, has been considered in
\cite{kurylev2010rigidity}. In~\cite{kurylev2010rigidity} the authors
show that broken geodesic data determine the boundary distance data
and use then the results of \cite{Katchalov2001, kurylev1997multidimensional} to prove that the broken geodesic data determine the Riemannian manifold up to an isometry. In~\cite{de2020foliated} we utilized Theorem~\ref{Th:smooth} of the current paper and generalized the result of~\cite{kurylev2010rigidity} on reversible Finsler manifolds $F(x,y)=F(x,-y)$, satisfying a convex foliation condition. 

We let $u$ be the solution of the Riemannian wave equation with a point
source at $(s,x) \in (0,\infty )\times M$. In
\cite{duistermaat1996fourier, greenleaf1993recovering} it is shown
that the image, $\Lambda$, of the wavefront set of $u$, under the
canonical isomorphism $T^\ast M \ni (x,p) \mapsto g^{ij}(x)p_i \in TM$, coincides
with the image of the unit sphere $S_xM$ at $x$ under the geodesic
flow of $g$. Thus $\Lambda \cap \p(S M)$, where $SM$ is the unit sphere bundle of $(M,g)$, coincides with the exit
directions of geodesics emitted from $p$. In
\cite{lassas2018reconstruction} the authors show that if $(M,g)$ is a
compact smooth non-trapping Riemannian manifold with smooth strictly
convex boundary, then generically the \textit{scattering data of point
  sources} $\{\p M, R_{\p M}(M)\}$ determine $(M,g)$ up to
isometry. Here, $R_{\p M}(x) \in R_{\p M}(M), \: x \in M$ stands for
the collection of tangential components to boundary of exit directions
of geodesics from~$x$ to~$\p M$.

A classical geometric inverse problem, that is closely related to the distance functions, asks: Does the Dirichlet-to-Neumann mapping of a Riemannian wave operator determine a Riemannian manifold up to isometry? For the full boundary data this problem was solved originally in~\cite{belishev1992reconstruction} using the Boundary control method. Partial boundary data questions have been studied for instance in \cite{lassas2014inverse, milne2016codomain}. Recently~\cite{kurylev2018inverse} extended these results for connection Laplacians. Lately also inverse problems related to non-linear hyperbolic equations have been studied extensively  \cite{kurylev2014inverse, lassas2018inverse, wang2016inverse}. For a review of inverse boundary value problems for partial differential equations see \cite{LassasICM2018, uhlmann1998inverse}.

Another well studied geometric inverse problem formulated with the distance functions is the Boundary rigidity problem. This problem asks:  Does the \textit{boundary distance function} $d_F\colon\p M \times \p M \to \R$, that gives a distance between any two boundary points, determine $(M,F)$ up to isometry? In an affirmative case $(M,F)$ is said to be boundary rigid. For a general Riemannian manifold the problem is false: Suppose the manifold contains a domain with very slow wave speed, such that all the geodesics starting and ending at the boundary avoid this domain. Then in this domain one can perturb the metric in such a way that the boundary distance function does not change. It was conjectured in~\cite{michel1981rigidite} that for all compact simple Riemannian manifolds the answer is affirmative. In two dimensions it was solved in~\cite{pestov2005two}. For higher dimensional case the problem is still open, but different variations of it has been considered for instance in \cite{burago2010boundary, croke1991rigidity, stefanov2016boundary, stefanov2017local}. In contrast to a Riemannian case a simple Finsler manifold is never boundary rigid~\cite{ivanov2013local}. The Boundary distance data~\eqref{eq:data}, studied in this paper, are much more data than the knowledge of the boundary distance function. Therefore we can obtain the optimal determination of $(M,F)$, as explained in theorems~\ref{Th:smooth} and~\ref{Th:analytic}, even though we pose no geometric conditions on $(M,F)$.



\subsubsection{Finsler and Riemannian geometry}
   
We refer to the monographs \cite{ bao2012introduction, shen2001lectures} for the development of Finsler manifolds without
boundaries. We point out that two major differences occur between the
Riemannian and Finslerian realms that are related to the proof of
Theorem~\ref{Th:smooth}. In Riemannian geometry the relation between
$TM$ and $T^\ast M$ is simple; raising and lowering indices provides a
fiberwise linear isomorphism. In Finsler geometry this is not
possible, since the Legendre transform (see~\eqref{eq:Legendre}) is
not linear in the fibers. For this reason, we have to be more careful
in the analysis of distance functions and the connection of their
differential to the velocity fields of geodesics. Moreover the
Finslerian gradient is not a linear operator.

The second issue arises from the lack of a natural linear connection
compatible with $F$ on vector bundle $\pi\colon TM \to M,$ where $\pi(x,y)=x$ is the canonical projection to the base point. In Section
\ref{Se:Appendix2} we consider properties of Chern connection
$\nabla$, which is a torsion free linear connection on the pullback
bundle $\pi'\colon \pi^\ast TM \to TM$ (see \cite{bao1996notable, chern1943euclidean, chern1996finsler}). We derive natural
compatibility relations for $\nabla$ and the fundamental tensor field
$g$ on $\pi^\ast TM$ (see \cite[Section 5.2]{shen2001lectures} and
Lemma~\ref{Le:Chern_prop_2}) in a special case. In
\cite{shen1994connection} Shen proved the general version of the
compatibility relations. We use Lemma~\ref{Le:Chern_prop_2} to
formulate the initial conditions for so-called transverse vector
fields (see \cite[Section III.6]{Chavel}) with respect to $\p M$ along
boundary normal geodesics. After this we give a definition of an index
form related to these vector fields and use it to prove results
similar to classical theorems, originally by Jacobi, related to the
minizing of geodesics after focal points (for the Riemannian case, see
for instance \cite[Section III.6]{Chavel}).

Inverse problems arising from elastic equations have been also extensively studied. See e.g.  \cite{BAO2018263, bal2015reconstruction, bao2018inverse, griesmaier2018uncertainty, hu2015nearly, nakamura1993identification, nakamura1994global}. 

\section{From Elasticity to Finsler geometry}
\label{Se:elastic}

The main physical motivation of this paper is to obtain a geometric and coordinate invariant point of view to the inverse problems related to the propagation of seismic waves. The seismic waves are modelled by the anisotropic elastic wave equation in $\R^{1+3}$. This elastic system can be microlocally decoupled to 3 different polarizations~\cite{stolk2002microlocal}. In this section, we introduce a connection between the fastest polarization (known as the primary polarization and denoted by qP) and the Finsler geometry. Moreover it turns out that the Finsler metric arising from elasticity is fiberwise real analytic. We use the typical notation and terminology of the seismological literature, see for instance
\cite{cerveny2005seismic}. We let $c_{ijk\ell}(x)$ be the smooth
stiffness tensor on $\R^3$ which satisfies the symmetry
\begin{equation}
\label{eq:symmetry_of_elastic_tensor}
c_{ijk\ell}(x)=c_{jik\ell}(x)=c_{k\ell ij}(x), \quad x\in \R^3.
\end{equation}
We also assume that the density $\rho(x)$ is a smooth function of $x$ and define density--normalized elastic moduli
\begin{equation}
a_{ijk\ell}(x)=\frac{c_{ijk\ell}(x)}{\rho(x)}.
\end{equation}
The elastic wave operator $P$, related to $a_{ijk\ell}$, is given by
\begin{equation}
P_{i\ell}=\delta_{i\ell}\frac{\p^2}{\p t^2}-a_{ijk\ell}(x)\frac{\p}{\p x^j}\frac{\p}{\p x^k}+\hbox{lower order terms.}
\end{equation}

For every $(x,p)\in \R^3\times \R^3$ we define a square matrix $\Gamma(x,p)$, by
\begin{equation}
\label{eq:Chirstoffel_matrix}
\Gamma_{i\ell}(x,p):=
\sum_{j,k}a_{ijk\ell}(x)p_jp_k
.
\end{equation}
The matrix $\Gamma(x,p)$ is called the \textit{Christoffel matrix}. Due to~\eqref{eq:symmetry_of_elastic_tensor} the matrix $\Gamma(x,p)$ is symmetric. One also assumes that $\Gamma(x,p)$ is positive definite for every $(x,p)\in \R^3\times (\R^3 \setminus \{0\})$. 

The principal symbol $\delta(t,x,\omega, p)$  of the operator $P$ is then given by
\begin{equation}
\delta(t, x,\omega, p)=\omega^2I-\Gamma(x,p), \quad (t, x,\omega, p)\in \R^{1+3}\times \R^{1+3}. 
\end{equation}
Since the matrix $\Gamma(x,p)$ is positive definite and symmetric, it has three positive eigenvalues $\lambda^m(x,p),\: m \in \{1,2,3\}$. 

For example in homogeneous medium, where $a_{ijkl}$ are constants, and there are no lower order terms on the operator $P$ and $\lambda^m(x,p)=\lambda^m(p),$ the equation
\begin{equation}
Pu(x,t)=0
\end{equation}
has plane wave solution $u(x,t)=v e^{i(p\cdotp x-(\lambda^m(p))^{1/2}t}$. These waves correspond to waves that propagate with speed $c^m(p)=(\lambda^m(p))^{1/2}$ to the direction $p$.

We assume that 
\begin{equation}
\label{eq:nonvanishing_lamda_der}
\lambda^1(x,p) > \lambda^{m}(x,p), \quad  m \in \{2,3\}\hbox{, $(x,p) \in \R^3\times (\R^3\setminus \{0\})$}. 
\end{equation}
Then it follows from the  Implicit Function Theorem that $\lambda^1(x,p)$ and a related unit eigenvector $q^1(x,p)$ are smooth with respect to $(x,p)$.  See for instance \cite[Chapter 11, Theorem 2]{Evans} for more details.  Moreover the function $\lambda^1(x,p)$ is homogeneous of degree $2$ with respect to $p$. Due to~\eqref{eq:Chirstoffel_matrix} the Christoffel matrix $\Gamma(x,p)$ is real analytic on each fiber $T^\ast_x \R^3$.
Since $\lambda^1(x,p)$ is defined by a polynomial equation $\det{(\Gamma(x,p)-\lambda(x,p) I)}=0$ in $p$, we may apply the real analytic implicit function theorem (see e.g.~\cite[Theorem 2.3.5]{SP:real-analytic}) to see that $\lambda^1$ is real analytic on the fibers of~$T^\ast \R^3$.

\medskip

To keep the notation simple, we write from now on $\lambda:=\lambda^1(x,p)$ and $q:=q^1(x,p)$. We use $\Gamma q = \lambda q$ and~\eqref{eq:Chirstoffel_matrix} to compute the Hessian of $\lambda(x,p)$ with respect to $p$ to obtain
\begin{equation}
\label{eq:Hessian_of_g^1_1}
\hbox{Hess}_p(\lambda(x,p))= 2\bigg(\Gamma(q(x,p))+(Dq)^T(\lambda(x,p)I-\Gamma(x,p))Dq\bigg),
\end{equation}
where $Dq$ is the Jacobian of $q(x,p)$ with respect to $p$ and the
superscript $T$ stands for transpose. Since $\Gamma(q(x,p))$ is
positive definite it follows from~\eqref{eq:nonvanishing_lamda_der}
and~\eqref{eq:Hessian_of_g^1_1} that the Hessian of $\lambda(x,p)$ is
also positive definite. We note that a similar result has been
presented in~\cite{antonelli2003geometrical} under the assumption the
stiffness tensor is homogeneous and transversely isotropic.

\medskip

We define a continuous function $f(x,p):=\sqrt{\lambda(x,p)}$, which
is smooth outside $\R^3 \times \{0\}$. We conclude with summarizing
the properties of the function $f$:


\begin{equation}
\label{eq:co_fins}
\begin{split}
\text{(i)}&\hbox{ The function $f \colon \R^3\times(\R^3\setminus \{0\})\to (0,\infty)$ is smooth,}
\\
&\hbox{ \: real analytic on the fibers}
\\
\text{(ii)}&\hbox{ For every $(x,p) \in \R^3\times \R^3$ and $s\in \R $ }
\\
&\hbox{ \: it holds that $f(x,sp)=|s|f(x,p)$}
\\
\text{(iii)}&\hbox{ For every $(x,p) \in \R^3\times(\R^3\setminus \{0\})$ the Hessian of $\frac{1}{2}f^2$ is}
\\
&\hbox{ \: symmetric and positive definite with respect to $p$.}
\end{split} 
\end{equation}
Therefore, $f$ is a convex (Minkowski) norm on the cotangent space. Finally, we define a Finsler function $F$ to
be the Legendre transform of $f$. Thus the bicharacteristic curves of
Hamiltonian $\frac{1}{2}\big(\lambda(x,p)-1\big)$ are given by the
co-geodesic flow of $F$. Moreover the $qP$ group velocities are given
by the Finsler structure.

\medskip

Another geometrical inverse problem on Finsler manifolds, using
exterior geodesic sphere data, is presented in~\cite{Finsler_Dix},
extending an earlier result on Riemannian manifolds~\cite{deHoop1}.

\medskip 

The connection between Finsler geometry elastic media has been also considered for instance in \cite{ clayton2015finsler, clayton2017finsler, yajima2009finsler}. For Finsler geometry and the deformations of oriented media we suggest the reader to see for instance~\cite{bejancu1990finsler}.

\subsection{Inverse problem for the travel times of qP-waves}

The travel time data that we consider next can be viewed as being obtained from propagation of singularities of \textit{qP} polarized waves generated at Dirac sources in the interior. This originates from the work of Dencker~\cite{dencker1982propagation}, and follows upon microlocally diagonalizing the elastic wave operator with a smooth stiffness tensor, even though it is not of principal type. Nonetheless, the diagonal component $\lambda$, which is second-order elliptic pseudodifferential in space, associated with \textit{qP} polarized waves can be smoothly extracted. (This is not the case for the other polarizations.) We note that the $\Psi$DO with principal symbol  $\omega^2-\lambda(x,p), \: \omega\neq 0$ is of real principal type on $\R^{1+3}$, and one can carry out a parametrix construction for this operator, which generates a classical Fourier integral operator. The principal symbol $\omega^2-\lambda(x,p)$ of the pseudodifferential operator defines a Hamiltonian and the associated flow determines a Lagrangian submanifold and its corresponding generating function. This function yields the phase function in the oscillatory integral representation of the parametrix. By H\"{o}rmander's theorem \cite[Theorem 2.5.15]{hormander1971fourier}, the propagation of singularities by the mentioned Fourier integral operator is determined by this phase function, and by the implication of \cite[Proposition 2.1]{greenleaf1993recovering} follows the Hamiltonian flow. Travel time, arrival of singularities, is measured along the Hamiltonian flow. The wavefront set detection, at the boundary, of the polarized wave solution, generated by each Dirac source in the interior thus provides the data in our inverse problem. Moreover, the projection of the Hamiltonian flows on the base manifold of the cotangent bundle coicide with the geodesics in the Finsler geometry, where the co-Finsler function is identified with the principal symbol of the above mentioned pseudodifferential operator $\lambda$.

Let $M \subset \R^3$ be a bounded domain with a smooth boundary~$\p M$. We assume that~$\p M$ is strictly convex with respect to~$F$, that is the second fundamental form of $\p M$ is positive definite. In this case the compact Finsler manifold $(\overline M,F)$ is geodesically complete~\cite{bartolo2011convex}, which means that any pair of points can be connected by a distance minimizing geodesic. We suppose that~$F$ is a complete Finsler metric on~$\R^3$. Moreover we assume that the shortest curves of $(\R^3,F)$ connecting points in~$\overline M$ are contained in~$\overline M$. We call these  the ``completeness" conditions for $(\overline M,F)$.

Let $x_0 \in M$, $t_0>0$ and let $h(t,x)$ be a Dirac delta function at $(t_0,x_0) \in \R\times M$. We also denote $\mathrm D:= \mathrm i(\p_{x_1}, \p_{x_2}, \p_{x_3})$. Let $u^{t_0,x_0}$ be the solution of $(\frac{\p^2}{\p t^2}-\lambda(x,\mathrm{D}))u(t,x)=h(t,x)$ on $(0,\infty) \times \R^3$, with vanishing initial conditions. We note that
$
\p (\R \times M)=\R\times \p M,
$
and therefore for any $(x,t,\omega,p) \in \p (T^\ast(\R \times M))$ it holds that $(x,p)\in \p (T^\ast M)$ and $x \in \p M$. Then we set
\begin{equation}
\Lambda_{x_0,t_0}:= \p (T^\ast(\R \times M)) \cap \hbox{wavefront set}(u^{t_0,x_0})
\end{equation}
and define the \textit{travel time} from $x_0$ to $z \in \p M$ to be
\begin{equation}
\mathcal T(t_0,x_0,z):=\inf \{t>0: \: T^\ast_{t,z} (\R \times M) \cap \Lambda_{x_0,t_0} \neq \emptyset \}-t_0.
\end{equation}

\begin{Theorem}
\label{Th:elastic}
Let $c_{ijk\ell}(x)$ be an ansitropic elastic tensor and~$\rho$ a smooth density of the mass on $\R^3$. Suppose that the eigenvalues of the corresponding Christoffel matrix $\Gamma$ satisfy the inequality~\eqref{eq:nonvanishing_lamda_der}. Then the square root of the largest eigenvalue of~$\Gamma$ satisfies (i)--(iii) of~\eqref{eq:co_fins}.

Let $M \subset \R^3$ be an open bounded set with smooth boundary. If $F$ is the Finsler metric on $M$ given by $\lambda^1$ that satisfies the completeness conditions, stated above, on $M$ then the travel time data
\begin{equation}
\{\{\mathcal T(t_0,x_0,z): \: z \in \p M\}, \: x_0 \in M, \: t_0>0\}
\end{equation}
determine the isometry class of $(\overline M,F)$.
\end{Theorem}
\begin{proof}
By the completeness assumptions it holds that $\mathcal T(x_0,z)=d_F(x_0,z)$ for any $x_0 \in M$ and $z \in \p M$. Thus the  result follows from Theorem~\ref{Th:analytic}.
\end{proof}

\medskip


\section{Proof of theorem~\ref{Th:smooth}}
\label{Se:proof}

In this section we provide a proof of Theorem~\ref{Th:smooth}. The
proof is divided into four parts. In the first part, we consider the
topology and introduce a homeomorphism $\Psi$ from $(M_1,F_1)$ onto
$(M_2,F_2)$. The second part is devoted to proving that homeomorphism
$\Psi$ is smooth and has a smooth inverse. In the third part, we study
smoothness of a distance function $d_F(\cdot,z), \: z \in \p M$ in
those interior points $x$ where a distance minimizing curve from $x$
to $z$ is a geodesic contained in the interior. Then, in the final
part, we use the result obtained in the third part to prove that the
Finsler functions $F_1, \: \Psi^\ast F_2$ coincide in the set
$\overline{G(M_1, F_1)}$ (recall Notation~\ref{De:good_set}), but
not necessarily in its exterior.

\subsection{Topology}
Here, we define a map $\Psi\colon(M_1,F_1)\to (M_2,F_2)$ that will be
shown to satisfy the claim of Theorem~\ref{Th:smooth}. Whenever we do
not need to distinguish manifolds $M_1$ and $M_2$ we drop the
subindices.

We start with showing that data~\eqref{eq:data} determine the function $r_x\colon\p M \to \R$ for each $x \in \p M$. By the triangle inequality and the continuity of distance function $d_F(\cdot,z)$ on $M$ we have
\begin{equation}
\label{eq:boundary_point_dist}
r_x(z):=d_F(x,z)=\sup_{q \in M^{int}}(d_F(q,z)-d_F(q,x))=\sup_{q \in M^{int}}(r_q(z)-r_q(x))
\end{equation}
for all  $z \in \p M$. Thus data~\eqref{eq:data} determine $r_x$, moreover~\eqref{eq:BDD_agree_start} and~\eqref{eq:boundary_point_dist} imply
\begin{equation}
\label{eq:BDD_agree_1}
\{r_{x_1}:{x_1}\in M_1\}=\{r_{x_2}\circ \phi:{x_2}\in M_2\}\subset C(\p M_1).
\end{equation}


\color{black}
Since $\p M$ is compact it holds that for any ${x}\in M$ the corresponding boundary distance function $r_{x}$ belongs to $C(\p M) \subset L^\infty(\p M)$.  By~\eqref{eq:data} and~\eqref{eq:boundary_point_dist}  we have recovered the mapping
\begin{equation}
\label{eq:map_R}
\mathcal{R}\colon M \to C(\p M), \quad \mathcal{R}(x)=r_x.
\end{equation}
\color{black}
In the next proposition, we study the properties of this map.

\begin{Proposition}
\label{Pr:topology}
Let  $(M,F)$ be a smooth compact Finsler manifold with smooth boundary. The map $\mathcal{R}$ given by~\eqref{eq:map_R} is a topological embedding.
\end{Proposition} 
\begin{proof}
Since $M$ is compact, $d_F$ is a complete non-symmetric (path) metric, and by \cite[Theorem 2.5.23]{burago2001course} for any $x_1,x_2 \in M$ there exists a distance minimizing curve $\gamma\colon [0,d_F(x_1,x_2)]\to M$ from $x_1$ to $x_2$. Moreover, whenever $a,b \in [0,d_F(x_1,x_2)]$ are such that $\gamma((a,b))\subset M^{int}$, then  $\gamma\colon [a,b]\to  M$ is a geodesic. 


Since the unit sphere bundle $SM:=F^{-1}\{1\}$ is compact there exists a universal constant $L>1$,  such that for all $ x_1,x_2 \in M$ we have
\begin{equation}
\label{eq:quasi}
\frac{1}{L} d_F(x_1,x_2)\leq d_F(x_2,x_1) \leq L d_F(x_1,x_2).
\end{equation}
We let $x_1,x_2 \in M$ and $z\in \p M$. By triangular inequality we have
\begin{equation}
|d_F(x_1,z)-d_F(x_2,z)| \leq L d_F(x_1,x_2).
\end{equation}
Thus $\|r_{x_1}-r_{x_2}\|_\infty\leq L d_F(x_1,x_1)$, which proves that the map $\mathcal{R}$ is continuous.

We suppose then that $r_{x_1}=r_{x_2}$ for some $x_1,x_2 \in M$.
We let $z$ be one of the closest boundary points to $x_1$.
Then $z$ is also a closest boundary point to $x_2$.
Denote $r_{x_1}(z)=h$.
If $h=0$ then $x_1=z$ and thus $x_1=x_2$.
We suppose then that $h>0$, which means that $x_1$ and $x_2$ are interior points of $M$.

It follows from Lemma~\ref{Le:normal_geo_is_mini} that for both $i=1,2$ the minimizing curve from $x_i$ to $z$ is a geodesic and meets the boundary orthogonally.
Let us denote this geodesic by $\gamma_i$.
We shift time parameter so that $z=\gamma_1(0)=\gamma_2(0)$, and thus $x_i=\gamma_i(-h)$ for both indices.
As the two geodesics share the same initial point and initial direction, we conclude that $\gamma_1=\gamma_2$ and consequently 
\begin{equation}
x_1=\gamma_1(-h)=\gamma_2(-h)=x_2.
\end{equation}

The injectivity of $\mathcal{R}$ implies that it is a topological embedding, as any continuous mapping from a compact space to a Hausdorff space is closed.
\end{proof} 

Next we define maps
\begin{equation}
\Phi\colon C(\p M_1) \to C(\p M_2), \quad \Phi(f)=f\circ \phi^{-1}
\end{equation}
and 
\begin{equation}
\label{eq:map_Psi}
\Psi\colon M_1\to M_2, \quad \Psi=\mathcal R^{-1}_2\circ \Phi \circ \mathcal R_1
\end{equation}
Here $\mathcal R_i$ is defined as $\mathcal{R}$ in~\eqref{eq:map_R}. The main theorem of the section is the following

\begin{Theorem}
\label{Th:topology}
Let $(M_i,F_i), \: i =1,2$ be as in Theorem~\ref{Th:smooth}. Then the map $\Psi\colon M_1\to M_2$ given by~\eqref{eq:map_Psi} is a homeomoprhism. Moreover $\Psi|_{\p M_1}=\phi$.
\end{Theorem}

\begin{proof}
By~\eqref{eq:BDD_agree_1} and Proposition~\ref{Pr:topology} it holds that $\Psi$ is well defined. Clearly the map $\Phi$ is a homeomorphism and therefore $\Psi$ is a homeomorphism. 

We let $x_1 \in \p M_1$. Then $(\Phi \circ \mathcal R_1)(x_1)$ is $r_{x_2}$ for some $x_2 \in M_2$. Since 
\begin{equation}
r_{x_2}(\phi(x_1))=[(\Phi \circ \mathcal R_1)(x_1)](\phi(x_1))=r_{x_1}(x_1)=0,
\end{equation}
we have $d_{F_2}(x_2,\phi(x_1))=0$ and so $x_2=\phi(x_1)$.
This proves $\Psi(x_1)=\phi(x_1)$.
\end{proof}

%
%

\subsection{Differentiable structure}

Here, we show that the map $\Psi\colon M_1\to M_2$ is a
diffeomorphism. We split the study in two cases, near the boundary and
far from the boundary. We begin with the former one.

We extend $(M,F)$ to a closed Finsler manifold $(N,H)$ to facilitate the study of boundary points.
(This can be done for instance by constructing the ``double" of $M$. See for instance \cite[Example 9.32]{lee2013smooth}.)

We let $\stackrel{\leftarrow}{\nu_{in}}$ be the inward pointing unit normal vector field to $\p M$ with respect to reversed Finsler function $\stackrel{\leftarrow}{F}$. We define the
normal exponential map $\exp^\perp\colon \partial M\times\R\to N$ so that
\begin{equation}
\exp^\perp(z,s)
:= \;
\stackrel{\leftarrow}{\exp_z}(s\stackrel{\leftarrow}{\nu_{in}}(z)),
\end{equation}
where $\stackrel{\leftarrow}{\exp_z}$ is the exponential map of the reversed Finsler function $\stackrel{\leftarrow}{H}$.

\begin{Lemma}
\label{Le:boundary_normal_coordinates}
There exists $h>0$ such that $M$ is contained in the image of the normal map.
Moreover there exists $r>0$ such that $\exp^\perp\colon \p M\times [0,r) \to M$ is a diffeomorphism onto its image. 
\end{Lemma}
\begin{proof}
Define 
\begin{equation}
h=\max \{d_{F}(x,\p M): x\in M\}+c,
\end{equation}
for any $c>0$. Since $N$ is compact the map $\exp^\perp\colon \p M\times [0,h) \to N$ is well defined. Moreover it holds that any interior point can be connected to any of its closest boundary points via distance minimizing geodesic that is normal to the boundary. Therefore we conclude  that $M \subset \exp^\perp(\p M\times [0,h))$.

Notice that
\begin{equation}
\exp^\perp(z,t)=\pi(\stackrel{\leftarrow}{\phi}_t(z, \stackrel{\leftarrow}{\nu_{in}}(z))),
\end{equation}
where 
$ \stackrel{\leftarrow}{\phi_t}$ is the geodesic flow of $ \stackrel{\leftarrow}{H}$. Since $ \stackrel{\leftarrow}{\nu_{in}}$ is a smooth unit length vector field, this proves that $\exp^\perp$ is smooth. 

We let $z\in \p M$. Give any local coordinates $(z',f)$ near $z$ such that $f|_{\p M}=0$ is a boundary defining function. Then with respect to coordinates $(z',t)$ for $(\p M \times (-h,h))$ we have
\begin{equation}
D\exp^\perp (z,0)=
\left(\begin{array}{cc}
D_{z'}(z' \circ \exp^\perp) & \frac{\p}{\p t}(z'  \circ \exp^\perp)
\\
&
\\
D_{z'}(f  \circ \exp^\perp)& \frac{\p}{\p t}(f  \circ \exp^\perp)
\end{array} \right)
=
\left(\begin{array}{cc}
id_{n-1} & \overline{a}
\\
\\
\overline 0^T & df(\stackrel{\leftarrow}{\nu_{in}})
\end{array} \right),
\end{equation}
where $\overline{a}, \overline{0}\in \R^{n-1}$ and $df\big(\stackrel{\leftarrow}{\nu}_{in}\big)\neq 0$, since $\stackrel{\leftarrow}{\nu}_{in}$ is not tangential to $\p M$ and $f$ is a boundary defining function. Thus the Jacobian $\exp^\perp (z,0)$ is invertible and by the Inverse Function Theorem $\exp^\perp$ is a local diffeomorphism.  

Next we show that there exists $r\in (0,h)$ such that $\exp^\perp\colon \p M\times [0,r) \to M$ is a diffeomorphism onto its image.  If this does not hold, there exists a sequence $(x_j)_{j=1}^\infty \in M$ such that 
$$
\exp^\perp (z_j^1,s^1_j)=x_j=\exp^\perp (z_j^2,s^2_j)
$$
for some  $s^i_j \to 0$, $i\in \{1,2\}$ as $j\to \infty$ and for some boundary points $z^1_j$ and $z^2_j$ such that $(z_1,s_1)\neq (z_2,s_2).$
Then $d_F(x_j,\p M)\to 0$ as $j\to \infty$ and by the compactness of $N$ we may assume that $x_j\to x \in \p M$. Let $\epsilon >0$ and choose $j\in \N$ so that $d_F(x_j,x), s^i_j <\epsilon$. Then for $i\in \{1,2\}$ it holds that 
$$
d_F(x,z^i_j)\leq d_F(x,x_j)+d_F(x_j,z^i_j)<2L\epsilon
$$
where $L$ is the constant of~\eqref{eq:quasi}. Therefore, $z^i_j\to x$  as $j\to \infty$ for $i\in \{1,2\}$. This is a contradiction to the local diffeomorphism property of $\exp^\perp$. Thus there exists $r>0$ that satisfies the claim of this lemma.
\end{proof}

We immediately obtain the following.

\begin{Corollary}
\label{Le:boundary_coordinates}
Let $(M,F)$ be compact Finsler manifold with smooth boundary that is isometrically embedded into a closed Finsler manifold $(N,H)$.  Let us denote 
\begin{equation}
U(\p M, \epsilon):=\{x\in M: d_F(x,\p M)<\epsilon\}.
\end{equation}
There exists $\epsilon>0$ and a diffeomorphism $U(\p M, \epsilon) \ni x \mapsto (z(x),s(x)) \in (\p M \times [0,\epsilon))$, such that
\begin{equation*}
\label{eq:function_s}
d_F(x,z(x))=d_F(x,\p M)=s(x).
\end{equation*}
\end{Corollary}
\begin{proof}
The claim follows from Lemma~\ref{Le:boundary_normal_coordinates} by denoting $(z(x),s(x)):= (\exp^\perp)^{-1}(x)$.
%
%
\end{proof} 

\medskip
We then consider points far from the boundary. Our goal is to show that for every $x_0 \in M^{int}$ there exists points $(z_i)_{i=1}^n \subset \p M$ and a neighborhood $U$ of $x_0$ such that the map
\begin{equation}
U \ni x \mapsto (d_F(x,z_i))_{i=1}^n
\end{equation}
is a coordinate map. To do this we need to set up some notation. 

\begin{Definition}
\label{De:boundary_cut_points}
Let $z \in \p M$. We say that 
\begin{equation}
\tau_{\p M}(z):=\sup \{t>0: d_F(\exp^\perp(z,t),z)=d_F(\exp^\perp(z,t),\p M)=t\},
\end{equation}
is the \textit{boundary cut distance} to $z$. Then we define the collection of \textit{boundary cut points} $\sigma(\p M)$ as
\begin{equation}
\sigma(\p M)=\{\exp^\perp(z,\tau_{\p M}(z)): z \in \p M\}.
\end{equation}
\end{Definition}

The set $\sigma(\p M)$ is not empty and the next lemma explains why we cannot use the coordinate structure given by Lemma~\ref{Le:boundary_coordinates} far from $\p M$. 

%

\begin{Lemma}
\label{Le:prop_boundary_cut}
Let $z \in \p M$ and $t_0 =\tau_{\p M}(z)$. Then at least one of the following holds:
\begin{enumerate}
\item The map $\exp^\perp$ is singular at $(z,t_0)$.
\item There exists $q \in \p M, \: q \neq z$ such that $ \exp^\perp(z,t_0)= \exp^\perp(q,t_0)$.
\end{enumerate}
Moreover for any $t \in [0,t_0)$ the map $\exp^\perp$ is non-singular at $(z,t)$.
\end{Lemma}
\begin{proof}
The proof of the first claim is a modification of the proof of \cite[Chapter 13, Propostion 2.2]{do1992riemannian}.  
%
%
The proof of the last claim is  long. It is given in detail in Section~\ref{Se:Appendix2}.
\end{proof}

\begin{Lemma}
\label{Le:bcutlocus_func_cont}
The function $\tau_{\p M}\colon \p M \to \R$ is continuous. 
\end{Lemma}
The proof of Lemma~\ref{Le:bcutlocus_func_cont} is completely analogous to Klingenberg's lemma presented in \cite[Lemma 2.1.15]{klingenberg} or  \cite[Chapter 13, Proposition 2.9]{do1992riemannian}. 

The \textit{cut distance function} of the extended manifold $(N,H)$ is defined as 
\begin{equation}
\label{eq:cut_dist_func}
\tau(x,v)=\sup \{t>0:d_H(x,{\gamma}_{x,v}(t))=t\}, \quad (x,v)\in TN, \quad  F(x,v)=1.
\end{equation}
We call a point $\gamma_{x,v}(\tau(x,v))$ \textit{an ordinary cut point} to $x$.
In the next Lemma we show that a boundary cut point always occurs before an ordinary cut point.

\begin{Lemma}
\label{Le:boundary_vs_normal_cut_dist}
For any $z \in \p M$ it holds that
\begin{equation}
\stackrel{\leftarrow}{\tau}(z,\stackrel{\leftarrow}{\nu}_{in}(z))>\tau_{\p M}(z),
\end{equation}
where $\stackrel{\leftarrow}{\tau}$ is the cut distance function of the reversed Finsler metric $\stackrel{\leftarrow}{H}$.
\end{Lemma}

\begin{proof}
The claim follows by applying Klingenberg's lemma to the geodesic $\gamma_{p,\xi}$, where $p=\gamma_{p,\xi}(\tau_{\p M}(z))$ and $\xi=-\dot\gamma_{p,\xi}(\tau_{\p M}(z))$ that hits normally to the boundary at the point $z$, see \cite[Lemma 2.13]{Katchalov2001} for details. 
\end{proof}

\begin{Corollary}
\label{cor:39}
Let $x_1 \in M^{int}$ and $z_{x_1} \in \p M$ be a closest boundary point to $x_1$.
There exist neighborhoods $U \subset M$ of $x_1$ and $V\subset\partial M$ of $z_{x_1}$ such that for every $(x,z)\in U\times V$ there exists the unique distance minimizing unit speed geodesic $\gamma_{x,z}$ from $x$ to $z$ and moreover $\gamma_{x,z}([0,d_F(x,z)) \subset M^{int}$.
Both the distance $d_F(x,z)$ and the geodesic $\gamma_{x,z}$ depend smoothly on $(x,z)\in U\times V$.
\end{Corollary}

\begin{proof}
Recall that we have assumed the manifold $(M,F)$ being isometrically embedded in a closed Finsler manifold $(N,H)$.
By Lemmas ~\ref{Le:boundary_vs_normal_cut_dist} and \ref{Le:smoothness_of_the_distance_function} there are neighbourhoods $\tilde U \subset M^{int}$ of $x_1$ and $\tilde V\subset N$ of $z_1$ such that the distance function $d_H\colon \tilde U \times \tilde V\to \R$ is smooth.
Moreover for each $x \in \tilde U$ and $z \in \tilde V$ there exists a unique $H$-distance minimizing unit speed geodesic connecting $x$ to $z$.
This is denoted by $\gamma_{x,z}$. 

Let $y(x,z)\in S_xN$ be the initial velocity of the geodesic $\gamma_{x,z}$.
Since the exponential map is invertible at $(x_1,y(x_1,z_1))\in TN$ it holds that the map $(x,z)\mapsto y(x,z) \in SN$ is smooth in $\tilde U \times \tilde V$.
Therefore $\gamma_{x,z}$ depends smoothly on $(x,z)\in  \tilde U \times \tilde V$. 

As the geodesic between the reference points is normal to $\p M$ at $z_1$ the implicit function theorem yields that the function $t(x,y)$ of Notation \ref{De:good_set} is smooth in $SN$ near $y(x_1,z_1)$. This implies that
\begin{equation}
d_H(x,z)=t(x,y(x,z)),
\end{equation}
for any $z\in \p M$ near $z_1$. Since $H$ and $F$ agree on $TM$ it holds that $\gamma_{x,z}$, for any $z \in \p M$ near $z_1$, is also a geodesic of $(M,F)$ on the interval $[0,t(x,y)]$.  Therefore $\gamma_{x,z}[0,t(x,y(x,z)))\subset M^{int}$ and after possibly choosing smaller $U\subset \tilde U$ and $V\subset \tilde V \cap M$ we see that the $F$-distance and $H$-distances agree on $U \times V$. Clearly the map $(x,z)\mapsto y(x,z)$ is also smooth in this set.
This concludes the proof.
\end{proof}

We let $z \in \p M$ and define an evaluation function $E_z\colon \mathcal{R}(M) \to \R$ by $E_z(r)=r(z)$.
The functions $E_z$ correspond to the distance function
$d_F(\cdot,z)\colon M \to \R$ via the equation 
\begin{equation}
\label{eq:distance_func}
d_F(x,z)=(E_z \circ \mathcal{R})(x).
\end{equation}
Since $z \in \p M$ was an arbitrary point  we note that the function $d_F\colon M \times \p M\to \R$ is determined by the data~\eqref{eq:data} in the sense of~\eqref{eq:distance_func}.

\medskip

We define the \textit{exit time function}
\begin{equation}
\label{eq:exit_time_func}
\tau_{exit}\colon SM^{int} \to [0,\infty], \quad \tau_{exit}(x,v):=\inf\{t>0: \gamma_{x,v}(t) \in \p M\}.
\end{equation}
Comparing this to the function $t\colon TM\setminus0\to[0,\infty]$ defined in Notation~\ref{De:good_set}, we have $\tau_{exit}=t|_{SM^{int}}$.
The original data was formulated on the whole slit tangent bundle because it is independent of the Finsler metric, unlike the unit sphere bundle.

\begin{Lemma}
If $(x,v)\in SM^{int}$ is such that $\tau_{exit}(x,v)<\infty$ and $\dot{\gamma}_{x,v}(\tau_{exit}(x,v))$ is transversal to
$\partial M$ then there exists a neighborhood $U\subset SM$ of $(x,v)$ such that $\tau_{exit}|_{U}$ is well defined  and $C^\infty$-smooth.
\end{Lemma}
\begin{proof}
Since  $\dot{\gamma}_{x,v}(t_0)$ is not tangential to $\p M$ the claim follows from the 
Implicit Function Theorem in boundary coordinates. 
\end{proof}

Take an interior point $x\in M$ near which we want to construct a
system of coordinates. We let $v\in S_xM$ be such that the geodesic
$\gamma_{x,v}$ emanating from $x$ to the direction $v$ is the shortest
curve between $x$ and a terminal boundary point $z_x$. By Lemma~\ref{Le:boundary_vs_normal_cut_dist} these two points are not conjugate along $\gamma_{x,v}$.   

We let $U \subset SM$ be so small neighborhood of $(x,v)$ that the exit time function $\tau_{exit}\colon  U \to \R$ is defined and smooth. We have thus assumed that $x$ and
$z_x=\gamma_{x,v}(\tau_{exit}(x,v))$ are connected minimally and without conjugate points
by $\gamma_{x,v}$.

We let $\ell_x\colon T_xM\to T_x^*M$ be the Legendre
transform, (to recall the definition see~\eqref{eq:Legendre} in the appendix). It and its inverse are smooth outside the origin. Thus the distance function $d_F(\cdot,z_x)$  is smooth near $x$ and its differential at $x$ is $\ell_x(v)\in T_x^*M$ (see Lemma~\ref{Le:differential_of_dist_func}).

Pick any $u\in T_x^*M\setminus \{0\}$ with $\langle u,v\rangle=0$. For
$s\in\mathbb R$, denote 
\begin{equation}
v_s=\frac{\ell_x^{-1}(\ell_x(v)+su)}{F^\ast(\ell_x(v)+su)}.
\end{equation}
Here $F^\ast$ is the dual of $F$, (see~\eqref{eq:F_dual}). The map $s\mapsto v_s\in S_xM$ is smooth.

Consider the geodesics $\gamma_{v_s}$ starting at $x$ in the direction
$v_s$. Since $\gamma_{x,v}(\tau_{exit}(x,v))$ is transversal to $\partial M$, then
$s\mapsto\gamma_{v_s}(\tau_{exit}(x,v_s))$ is smooth near $s=0$. Also since $x$ is not an ordinary cut point to $\gamma_{v_s}(\tau_{exit}(x,v_s))$ at $s=0$, it is not
either an ordinary cut point to $\gamma_{v_s}(\tau_{exit}(x,v_s))$ when $\lvert s\rvert$ is small. Therefore, for $s$
sufficiently close to zero the distance function to
$\gamma_{v_s}(\tau_{exit}(x,v_s))$ is smooth near $x$.

The differential of the distance function 
at $x$ amounts to
\begin{equation}
\ell_x(v_s)=\frac{\ell_x(v)+su}{F^\ast(\ell_x(v)+su))}.
\end{equation}
Therefore, for any $u$ with the required property there is a small non-zero $s$ so that there is a distance function to a boundary point which is smooth near $x$ and the differential at $x$
is $\frac{\ell_x(v)+su}{F^\ast(\ell_x(v)+su)}$.

We take $n-1$ covectors $u_1,\dots,u_{n-1} \in T^\ast_x M$ so that the set
\begin{equation}
\{\ell_x(v),u_1,\dots,u_{n-1}\}\subset T_x^*M
\end{equation}
is linearly independent
and each $u_i\in T_x^*M$ is orthogonal to $v\in T_xM$. For each
$i=1,\dots,n-1$ we take $s_i\neq0$ so that 
\begin{equation}
\frac{\ell_x(v)+s_iu_i}{F^\ast(\ell_x(v)+s_iu_i)}
\end{equation}
 is the differential of a distance function to a boundary point as described
above.

This gives rise to distance functions to $n$ boundary points close to
one another. These functions are smooth near $x$ and the differentials
are 
\begin{equation}
\ell_x(v),\frac{\ell_x(v)+s_1u_1}{F^\ast(\ell_x(v)+s_1u_1)}, \ldots, \frac{\ell_x(v)+s_{n-1}u_{n-1}}{F^\ast(\ell_x(v)+s_{n-1}u_{n-1})}.
\end{equation} 
This set is linearly independent, so the distance functions give a smooth system of
coordinates in a neighborhood of $x$. Thus we obtain

\begin{Lemma}
\label{Le:interior_coordinates}
Let $x_0 \in M^{int}$. There is a
neighborhood $U$ of $x_0$ and points $z_1,\ldots,z_n \in \p M$, where $z_1$ is a closest boundary point to $x_0$, so that the mapping $U \ni  x \mapsto (d_F(x,z_i))_{i=1}^n$ is a smooth coordinate map. 

Moreover there exists an open neighborhood $V \subset \p M$ of $z_1$ such that the distance function $d_F\colon U\times V\to \R$ is smooth and the set
\begin{equation}
\mathcal{V}:=
\bigg\{(z_i)_{i=2}^n \in V^{n-1}:  \det (f_{z_2,\ldots, z_n}(x))\bigg|_{x=x_0}\neq 0\bigg\}.
\end{equation}
is open and dense in $V^{n-1}:=V\times \cdots \times V$. Where 
\begin{equation}
\label{eq:determinant_of _distances}
f_{z_2,\ldots, z_n}(x):=D\widetilde f_{z_2,\ldots, z_n}(x),
\end{equation}
and $D\widetilde f_{z_2,\ldots, z_n}$ stands for the pushforward of the map 
\begin{equation}
\widetilde f_{z_2,\ldots, z_n}(x):=(d_F(x,z_i))_{i=1}^n)\in \R^n, \quad x \in U.
\end{equation}
\end{Lemma}

\begin{proof}
It remains to show that the set $\mathcal{V}$ is open and dense in $V^{n-1}$. 
Clearly the function 
\begin{equation}
G:V^{n-1} \to \R, \quad G(z_2,\ldots,z_n)=\det(f_{z_2,\ldots, z_n}(x_0))
\end{equation}
is continuous. Thus $\mathcal{V}=V^{n-1}\setminus G^{-1}\{0\}$ is open. Since the Legendre transform is an metric isometry between fibers, we have for every $z \in V$ 
\begin{equation}
d(d_F(\cdot,z))|_{x_0} \in S_{x_0}^\ast M:=\{p\in T_{x_0}^\ast M: F^\ast(p)=1\}
.
\end{equation}
(For the details, see Lemma~\ref{Le:differential_of_dist_func} in the appendix.)
We let $(e_i)_{i=1}^n$ be a basis of $T_{x_0}^\ast M$ and define a map $T\colon ( T_{x_0}^\ast M)^{n-1} \to \R$ by
\begin{equation}
T((u_i))_{i=2}^{n} =\det(M(\ell_{x_0}(v), u_2,\ldots, u_{n})),
\end{equation} 
where $\ell_{x_0}(v)$ is as in the discussion before this lemma and $M(\ell_{x_0}(v), u_1,\ldots, u_{n-1})$ is a real $n\times n$ matrix with columns $\ell_{x_0}(v), u_1,\ldots, u_{n-1}$, with respect to basis $(e_i)_{i=1}^n$ of $T_{x_0}^\ast M$. Notice that $(T_{x_0}^\ast M)^{n-1} $ is a real analytic manifold and $T$ is a multivariable polynomial, and thus a real analytic function. Moreover by the discussion before this lemma we know that $T$ is not identically zero. Therefore, it follows from  \cite[Lemma 4.3]{helgason2001differential} that $T^{-1}\{0\}\subset ( T_{x_0}^\ast M)^{n-1} $ is nowhere dense. Since  determinant is multilinear and the map $V\ni z\mapsto d(d_F(\cdot,z))|_{{x_0}} \in S_{x_0}^\ast M$ is a smooth embedding, it follows that $\mathcal{V} \subset V^{n-1}$ is dense.
\end{proof}

Now we are ready to prove the main theorem.

\begin{Theorem}
\label{Th:diffeo}
The mapping $\Psi$ given in~\eqref{eq:map_Psi} is a diffeomorphism.
\end{Theorem}
\begin{proof} We let $x_0 \in M_1$. 
We suppose first that $x_0$ is an interior point. By the data~\eqref{eq:BDD_agree_1} it holds that $z \in \p M_1$ is a minimizer of $d_{F_1}(x_0,\cdot)|_{\p M_1}$ if and only if $\phi(z)\in \p M_2$ is a minimizer of $d_{F_2}(\Psi(x_0),\cdot)|_{\p M_2}$. We let $z_1$ be a minimizer of $d_{F_1}(x_0,\cdot)|_{\p M_1}$. Since the map $\phi\colon \p M_1 \to \p M_2$ is a diffeomorphism it follows from the  Lemma~\ref{Le:interior_coordinates} that there exist points $z_2,\ldots,z_{n}\in \p M_1$ and a (suitably shrunk) neighborhood $U\subset M_1$ of $x_0$ such that the maps $U \ni x_1 \mapsto (d_{F_1}(x_1,z_i))_{i=1}^n$ and 
\begin{equation}
\Psi(U) \ni x_2 \mapsto (d_{F_2}(x_2,\phi(z_i)))_{i=1}^n=(d_{F_1}(\Psi^{-1}(x_2),z_i))_{i=1}^n
\end{equation}  
are smooth coordinate maps. Thus with respect to these coordinates it holds that the local representation of~$\Psi$ near~$x_0$ is an identity map of~$\R^n$. This proves that~$\Psi$ is a local diffeomorphism near any interior point of~$M_1$.

\medskip
We consider next the boundary case. We denote 
\begin{equation}
U=\{x_1 \in M_1: \hbox{ there exists precisely one minimizer for } d_{F_1}(x_1,\cdot)|_{\p M_1}\}^{int}.
\end{equation}
By~\eqref{eq:BDD_agree_1} and since $\Psi$ is a homeomorphism, it holds that
\begin{equation}
\label{eq:unique_boundary_distance_minimizers}
\Psi(U)=\{x_2 \in M_2: \hbox{ there exists precisely one minimizer for } d_{F_2}(x_2,\cdot)|_{\p M_2}\}^{int}.
\end{equation}
By Lemma~\ref{Le:boundary_normal_coordinates}  it follows that there exists $\epsilon>0$ and open neighborhood $V\subset U \subset M_1$ of $\p M_1$ such that the maps 
\begin{equation}
\begin{split}
&\: \exp^\perp_{F_1}\colon \p M_1 \times [0,\epsilon) \to V, \quad \hbox{and}
\\
& \: \exp^\perp_{F_2}\colon \p M_2 \times [0,\epsilon) \to \Psi(V) 
\end{split}
\end{equation}
are diffeomorphisms. Moreover due to Lemma~\ref{Le:boundary_coordinates} for every $x \in V$ it holds that $x=\exp^\perp_{F_1}(z(x),s(x))$, where $z(x)$ is the minimizer of $d_{F_1}(x, \cdot)|_{\p M_1}$ and $s(x)=d_{F_1}(x, z(x))$. Therefore, by~\eqref{eq:unique_boundary_distance_minimizers} it holds that 
\begin{equation}
( \exp^\perp_{F_2})^{-1}(\Psi(p))
=(\phi(z(p)), s(p)).
\end{equation}
Thus we have proved that with respect to coordinates $(V,( \exp^\perp_{F_1})^{-1})$ and
\\
$(\Psi(V),( \exp^\perp_{F_2})^{-1})$ the local representation of $\Psi$ is 
\begin{equation}
(\p M_1 \times [0,\epsilon))\ni(z,s)\mapsto (\phi(z),s) \in (\p M_2 \times [0,\epsilon)).
\end{equation}
Since $\phi\colon \p M_1 \to \p M_2$ is a diffeomorphism we have proved that $\Psi$ is a local diffeomorphism near $\p M_1$.

\medskip 
By Theorem~\ref{Th:topology} the map $\Psi$ is one-to-one and we have proved that $\Psi$ is a diffeomorphism. 
\end{proof}

\subsection{Smoothness of the boundary distance function}

Here we consider the smoothness of a boundary distance function and show that the closure of
\begin{equation}
\label{eq:good_set}
\widehat{G}(M,F):=\{ (x,y)\in G(M,F): x \in M^{int}, \: d_F(\cdot,z(x,y)) \hbox{ is $C^\infty$ at } x\} \cup \p_{out} T M ,
\end{equation}
where $\p_{out} T M \subset \p( TM)$ is the collection of outward pointing vectors (excluding tangential ones),
coincides with the set $\overline{G(M,F)}$ (see Definition~\ref{De:good_set}). This will be used in the next subsection to reconstruct $F$ in $G(M,F)$.

\medskip
We note that if $(x,y)\in G(M,F)$, with $F(y)=1$  then 
\begin{equation}
\tau_{exit}(x,y)=t(x,y),
\end{equation}
where $t(x,y)$ is as in Definition~\ref{De:good_set}. We assume below in this section that all vectors are of unit length. 

For those $(x,z)\in M^{int}\times \p M$ for which $d_F(\cdot,z)$ is smooth at $x$ we can use the differential  of the distance function to determine the image of the distance minimizing geodesic from $x$ to $z$. In this sense our problem is related to the Finslerian version of Hilbert's $4^{th}$ problem which is: To recover Finsler metric from the images of the geodesics. In this  setting the problem has been studied for instance in \cite{bucataru2016funk, bucataru2012projective,  bucataru2014finsler,  matveev2012projective, paiva2005symplectic}.
\medskip

The main result in this section is the following.

\begin{Proposition}
\label{Le:closure_of_Gs}
For any smooth connected and compact Finsler manifold $(M,F)$ with smooth boundary it holds that 
\begin{equation}
\overline{\widehat{G}(M,F)}=\overline{G(M,F)}.
\end{equation}
\end{Proposition}

We need a couple of auxiliary results to prove this proposition. We  state these auxiliary results below and prove them after the proof of Proposition~\ref{Le:closure_of_Gs}.

\begin{Lemma}
\label{Le:corner_in_the_curve_2}
Let $x_1,x_2\in M$ and let $c\colon [0,1]\to M$ be a rectifiable curve from $x_1$ to $x_2$. Let $t_0 \in (0,1)$ be such that $c(t_0)\in \p M$. If there exits  $\delta>0$ such that $c|_{[t_0-\delta, t_0]}$ is a geodesic and  
$\lim_{t \nearrow t_0} \dot{c}(t)$ is transversal to $\p M$ then there exists a rectifiable curve $\alpha\colon [0,1]\to M$ from $x_1$ to $x_2$ such that
\begin{equation}
\mathcal L (\alpha)<\mathcal{L}(c).
\end{equation}
\end{Lemma}

\begin{Lemma}
\label{Le:transverse_exit_direction_and_smooth_dist}
Suppose that $(x,v)\in G(M,F)$. If the exit direction is transversal to the boundary then for any $s\in (0,\tau_{exit}(x,v))$ the point $(x',v'):=(\gamma_{x,v}(s),\dot \gamma_{x,v}(s))\in \widehat{G}(M,F)$.
\end{Lemma}

\begin{Lemma}
\label{Le:final_condition}
Suppose that $(x,v)\in G(M,F)$, and the exit direction $\eta$ is tangential to the boundary at $z$. Assume that there exists $h>0$ such that for any $h'\in (0,h)$ the geodesic $\stackrel{\leftarrow} {\gamma}_{z,\xi_{h'}}\colon [0,\tau_{exit}(x,v)]\to M$ is well defined, where  
\begin{equation}
\xi_{h'}:=\frac{-\eta+h'\stackrel{\leftarrow} {\nu}_{in}}{\stackrel{\leftarrow} {F}(z,-\eta+h'\stackrel{\leftarrow} {\nu}_{in})}\in T_zM.
\end{equation} 

Then there exist sequences $(h_j)_{j=1}^\infty, \: (\epsilon_j)_{j=1}^\infty\subset \R$ such that $h_j, \epsilon_j>0$, $h_j,\epsilon_j\to 0$ as $j\to \infty$ and moreover the geodesic $\stackrel{\leftarrow} {\gamma}_{z,\xi_{h_j}}$ is a distance minimizing curve of $(M,\stackrel{\leftarrow} {F})$ from $z$ to $\stackrel{\leftarrow} {\gamma}_{z,\xi_{h_j}}(\tau_{exit}(x,v)-\epsilon_j)$ for any $j\in \N$ that is large enough.

\end{Lemma}
\begin{proof}[Proof of Proposition~\ref{Le:closure_of_Gs}] 
Since the sets $\widehat{G}(M,F)$ and $G(M,F)$ are conical it suffices to prove that 
\begin{equation}
\overline{\widehat{G}(M,F)} \cap SM =\overline{G(M,F)} \cap SM.
\end{equation}

\medskip
We first prove $\widehat{G}(M,F) \subset \overline{G(M,F)}$, which implies $\overline{\widehat{G}(M,F)} \subset \overline{G(M,F)}$. We let $(x,v) \in \widehat{G}(M,F)$. If $x \in M^{int}$, then clearly $(x,v)\in \overline{{G}(M,F)}$.  If $(x,v) \in \p_{out}TM$ then due to transversality of $v$ and $T_x\p M$ there exists $\epsilon>0$ such that for every  $t\in (0,\epsilon)$ we have
\begin{equation}
(\gamma_{x,v}(-t),\dot \gamma_{x,v}(-t)) \in G(M,F).
\end{equation}
Thus $(x,v) \in \overline{G(M,F)}$.

\medskip
Next we show that $G(M,F) \subset \overline{\widehat{G}(M,F)}$. We let $(x,v) \in G(M,F)$. 
Lemma~\ref{Le:transverse_exit_direction_and_smooth_dist}  implies that for any $s \in (0,\tau_{exit}(x,v))$ we have $(\gamma_{x,v}(s),\dot \gamma_{x,v}(s))\in \widehat{G}(M,F)$ if $\dot \gamma_{x,v}(\tau_{exit}(x,v))$ is transversal to $\p M$. This implies, $(x,v)\in \overline{\widehat{G}(M,F)}.$

Therefore, we assume that $( \gamma_{x,v}(\tau_{exit}(x,v)),\dot \gamma_{x,v}(\tau_{exit}(x,v))):=(z,\eta)$ is tangential to $\p M$. 
We let $(N,H)$ be a smooth complete Finsler manifold without boundary that extends $(M,\stackrel{\leftarrow} {F})$ and $\Pi\subset T_zN$ be the two dimensional vector subspace spanned by $\{\eta,\stackrel{\leftarrow} {\nu_{in}}\}$. If $a\in (0,\tau_{exit}(x,v))$ is small enough, then 
\begin{equation}
S(a):=\{\stackrel{\longleftarrow} {\exp_z}(w)\in N: w\in \Pi,  \:\stackrel{\leftarrow} {F}(z,w)<a \}
\end{equation}
is a $C^1$-smooth hyper surface of $N$ with a coordinate system given by $\eta$ and $\stackrel{\leftarrow} {\nu_{in}}$. 

We note that possible after choosing smaller $a$ the set $S(a)\cap \p M$ is given by a $C^1$-smooth graph $(s,c(s))\in S(a)$ such that for $s<0$ we have $c(s)<0$. This follows since $\p M$ is a smooth co-dimension 1 manifold and with respect to the coordinates $(\eta,\stackrel{\leftarrow} {\nu_{in}})$ of $S(a)$ we have $(-t,0)=\gamma_{x,v}(\tau_{exit}(x,v)-t)$ and $\gamma_{x,v}(\tau_{exit}(x,v)-t)$ does not hit $\p M$, if $t\in (0,a)$. Thus for any $h>0$ and $ t\in (0,a)$ the geodesic $\stackrel{\leftarrow} {\gamma}_{z,\xi_h}$ with
\begin{equation}
\xi_h:=\frac{-\eta+h\stackrel{\leftarrow} {\nu_{in}}}{\stackrel{\leftarrow} {F}(z,-\eta+h\stackrel{\leftarrow} {\nu_{in}})}
\end{equation}
of the Finsler function $H$ satisfies $\stackrel{\leftarrow} {\gamma}_{z,\xi_h}(t)\in M^{int}$.

The interval $[0,\tau_{exit}(x,v)-\frac{a}{2}]$ is compact and $\gamma_{x,v}([0,\tau_{exit}(x,v)-\frac{a}{2}])\subset M^{int}$.
Thus there is $r>0$ such that for all $t\in [0,\tau_{exit}(x,v)-\frac{a}{2}]$ we have $d_F(\gamma_{x,v}(t),\p M)< r$. 
Therefore the continuity of the exponential map implies that for any $h>0$ small enough and $t\in (0,\tau_{exit}(x,v)]$ we have $\stackrel{\leftarrow} {\gamma}_{z,\xi_h}(t)\in M^{int}$.

We note that Lemma~\ref{Le:final_condition} implies that for any $h,\epsilon>0$ that are small enough we have 
\begin{equation}
(x',v'):=\bigg(\stackrel{\leftarrow} {\gamma}_{z,\xi_h}(\tau(x,v)-\epsilon),-\dot{\stackrel{\leftarrow} {\gamma}}_{z,\xi_h}(\tau(x,v)-\epsilon)\bigg)\in G(M,F).
\end{equation}
Then Lemma~\ref{Le:transverse_exit_direction_and_smooth_dist} implies $(x',v')\in \overline{ \widehat G(M,F)}$.  Taking $h$ and $\epsilon$ to zero we finally obtain  $(x,v)\in \overline{ \widehat G(M,F)}$.
\end{proof}

\begin{proof}[Proof of Lemma~\ref{Le:corner_in_the_curve_2}]
Since $\lim_{t \nearrow t_0} \dot{c}(t)$ is transversal to $\p M$, there exists $\epsilon\in (0,\delta)$ such that $c|_{(t_0-\epsilon,t_0)}$ is a geodesic in $M^{int}$. We let $(\widetilde M, \widetilde F)$ be any compact Finsler manifold that extends $(M,F)$ and for which $c(t_0)$ is an interior point. Since $c_{|_{t_0-\epsilon,t_0}}$ is also  a geodesic of the extended manifold $(\widetilde M, \widetilde F)$, it follows from \cite[Proposition 11.3.1]{shen2001lectures} that there  exists $t_1\in (t_0-\epsilon,t_0)$ such that for $x:=c(t_1), \:z:=(c(t_0))$ we have 
\begin{equation}
(t_0-t_1)=d_{\widetilde F}(x,z)=d_{F}(x,z)=\tau_{exit}(x,\dot{c}(t_1)),
\end{equation} 
and the exponential map of $(\widetilde M, \widetilde F)$ is a $C^1$-diffeomorphism from 
\begin{equation}
\{y\in T_xM:  F(y)<2(t_0-t_1)\}
\end{equation}
onto a metric ball $B_{\widetilde F}(x,2(t_0-t_1))$ of $(\widetilde M, \widetilde F)$. 

Since $\lim_{t \nearrow t_0} \dot{c}(t)$ is transversal to $\p M$ it follows from the  Implicit Function Theorem that there exists a neighborhood $U\subset S_xM$ of $v:=\dot c(t_1)$ where the function $\tau_{exit}
$ is smooth and $\tau_{exit}(x,w)<2(t_0-t_1),$ whenever $w \in U$. We let 
\begin{equation}
C:=\{rw\in T_xM: w \in U,\: r\in [0,\tau_{exit}(x,w)]\}. 
\end{equation}
Then $\exp_x(C)$ contains an open neighborhood of $z$ in $M$. Since path $c$ is continuous there exists $t_2>t_0$ such that $\exp_x^{-1}(c(t_2))\in C$, and moreover
\begin{equation}
F(x,\exp_x^{-1}(c(t_2)))=d_{\widetilde F}(x,c(t_2))=d_{ F}(x,c(t_2)),
\end{equation}
since for any $w \in U$ the radial geodesic $\gamma_{x,w}$ of $(\widetilde M, \widetilde F)$ has the minimal length among all curves connecting $x$ to $\gamma_{x,w}(t), t\in (0,2(t_0-t_1))$.

If we denote by $\widetilde c$ the geodesic of $(\widetilde M, \widetilde F)$ that satisfies the initial condition $(\widetilde c(0),\dot{\widetilde c}(0))=(x,\dot{c}(t_1))$, then it leaves $M$ at $z$. Therefore, there exits a geodesic $\gamma$ of $(M,F)$ connecting $x=c(t_1)$ to $c(t_2)$ which satisfies
\begin{equation}
\mathcal{L}(c|_{[t_1,t_2]})>\mathcal{L}(\gamma).
\end{equation}
This implies the claim.
\end{proof}

\begin{proof}[Proof of Lemma~\ref{Le:transverse_exit_direction_and_smooth_dist}]
We note first that it follows from the  Implicit Function Theorem that there exists a neighborhood $U\subset SM$ of $(x,v)$ such that function $\tau_{exit}$ is smooth in $U$. Therefore, the mapping
\begin{equation}
U \ni (\widetilde x,w)\mapsto (z(\widetilde x,w),\eta(\widetilde x,w)):=(\gamma_{\widetilde x,w}(\tau_{exit}(\widetilde x,w)),\dot \gamma_{\widetilde x,w}(\tau_{exit}(\widetilde x,w)))
\end{equation}
is smooth and without loss of generality we may assume $\eta(\widetilde x,w)$ is transverse to~$\p M$ for any $(\widetilde x,w)\in U$. 

We let $(N, H)$ be any compact Finsler manifold without boundary extending $(M,F)$. We set $(x',v'):=(\gamma_{x,v}(s),\dot \gamma_{x,v}(s))$ and it follows that the points $x'$ and $z:=z(x,v)$ are not conjugate along $\gamma_{x,v}$. Therefore, the exponential map $\exp_{x'}$, of Finsler function $H$ is a diffeomorphism in a neighborhood $V \subset  T_{x'} N$ of $hv',$ for $h:=(\tau_{exit}(x,v)-s)$ onto some neighborhood of $z$ in $N$. Moreover 
\begin{equation}
h=d_F(x',z)\geq d_{H}(x',z).
\end{equation} 

To finish the proof, we show that we can find a smooth Finsler manifold $(\widetilde M, \widetilde F)$ so that 
$
M \subset \widetilde M \subset N, \: \widetilde F=H|_{\widetilde M}, \: z \in \widetilde M^{int}
$
and there exists a neighborhood $A \subset M^{int}$ of $x'$ so that 
\begin{equation}
d_F(\widetilde x,z)=\stackrel{\leftarrow}{\widetilde F}(z, \exp_z^{-1}(\widetilde x)), \quad \widetilde x \in A.
\end{equation}
Above the exponential map is given with respect to $\stackrel{\leftarrow}{\widetilde F}$. This implies $(x',v')\in \widehat{G}(M,F)$ and since $s\in (0,\tau_{exit}(x,v))$ was arbitrary we have $(x,v)\in \overline{\widehat{G}(M,F)}$.

We let $W_0$ be the image of $V$ under the orthogonal projection $y\mapsto \frac{y}{F(x',y)}$  on $S_{x'}N$. We let $r_0\in (0, d_F(x',\p M))$ be so small that for any $w \in S_{x'}M$ geodesic $\gamma_{x',w}|_{[0,r_0]}$ is a distance minimizer of $H$ and contained in $M^{int}$. In addition we define 
\begin{equation}
\Gamma:=\{\exp_{x'}(r_0 w)\in M^{int}: w\in\overline{(S_{x'}N) \setminus W_0}\}.
\end{equation}
Since this set is compact it follows from the triangle inequality that there exists $\epsilon_0>0$ which satisfies
\begin{equation}
\label{eq:contradic_dist}
r_0+d_F(\Gamma,z)\geq d_F(x',z)+\epsilon_0,
\end{equation}
as otherwise there would exist a $\stackrel{\leftarrow}{F}$-distance minimizing curve from $z$ to $x$ which is not $C^1$ at $x'$.

%
For $p\in N$ and $r>0$ we define
\begin{equation}
\stackrel{\leftarrow}{B}_{H}(p,r):=\{q\in N: d_H(q,p)< r\}.
\end{equation}
Since the points $x'$ and $z$ are not conjugate along $\gamma_{x,v}$ we can choose a neighborhood set  $W_1\subset W_0$ of $v'$,  and $2\epsilon_1< \epsilon_0$, $\delta >0$ such that 
\begin{equation}
\stackrel{\leftarrow}{B}_{H}(z,2\epsilon_1) \subset (\exp_{x'}((0, h+\delta) \times W_1)) 
\cap \stackrel{\leftarrow}{B}_{F}(z,\epsilon_0))
\end{equation}
and the geodesic $\gamma_{x,v}$ is the shortest curve from $x'$ to $z$ contained in $\exp_{x'}([0, h+\delta) \times W_1))$. 

We write $M_k:=M \cup \overline{ \stackrel{\leftarrow}{B}_{H}(z,k\epsilon_1)}$, for $k\in \{1,2\}$. Finally, we let $(\widetilde M, \widetilde F)$ be any smooth compact Finsler manifold with boundary such that
\begin{equation}
M_1 \subset \widetilde M \subset M_2 \quad \hbox{and} \quad \widetilde F=H|_{\widetilde M}.
\end{equation}

If $\beta$ is a distance minimizing curve of $(\widetilde M, \widetilde F)$ from $x'$ to $z$ it is a geodesic of $(M,F)$ for  $t <r_0$. Therefore, we have that $\beta=\gamma_{x,v}$ if $\dot{\beta}(0)\in W_1$. If $\dot \beta(0)\in (W_0\setminus W_1)$, then $\beta$ hits $\p \widetilde M$ transversally outside $\stackrel{\leftarrow}{B}_{H}(z,2\epsilon_1)$, which cannot happen due to Lemma~\ref{Le:corner_in_the_curve_2}.  If $\dot{\beta}(0) \in \overline{(S_{x'}N) \setminus W_0}$ then by~\eqref{eq:contradic_dist} we have
\begin{equation}
r_0+d_F(\Gamma,z)-\epsilon_0\geq d_F(x',z)\geq d_{\widetilde F}(x',z)\geq r_0+s_1+2\epsilon_1,
\end{equation}
where $s_1\geq 0$ is the time it takes to travel from~$\Gamma$ to $(\stackrel{\leftarrow}{B}_{H}(z,2\epsilon_1)\cap \widetilde M)$ along the curve~$\beta$. We note that $\beta(r_0+s_1)$ is contained in~$M$ which implies 
\begin{equation}
s_1+\epsilon_0 \geq d_F(\Gamma,z).
\end{equation} 
Thus we arrive at a contradiction $0 \geq 2\epsilon_1$, and we have proven that $\gamma_{x,v}$ is the unique distance minimizing curve of $(\widetilde M, \widetilde F)$ connecting $x'$ to~$z$.

\medskip
Since $x'$ and $z$ are not conjugate points along $\gamma_{x,v}$, the exponential map of the reversed Finsler function $\stackrel{\leftarrow}{\widetilde F}$ is a diffeomorphism from a neighborhood of $-h\eta \in T_{z}\widetilde M, \: \eta:=\eta(x,v)$ to a neighborhood of~$x'$. 
Thus  the local distance function 
\begin{equation}
q \mapsto \rho(q,z):=\stackrel{\leftarrow}{\widetilde F}(z,\bigg(\!\stackrel{\leftarrow}{\widetilde \exp}_{z}\!\bigg)^{-1}(q))
\end{equation}
is smooth near $x'$ and due to earlier part of this proof it coincides with $d_{\widetilde F}(\cdot,z)$ at~$x'$. 

We suppose that there exists a sequence $(x_j)_{j=1}^\infty\subset M$  that converges to $x'$ and for which it holds that
\begin{equation}
\label{eq:local_dist_func}
d_{\widetilde F}(x_j,z) < \rho(x_j,z).
\end{equation}
We let $\beta_j$ be a distance minimizing curve of $\stackrel{\leftarrow}{\widetilde F}$ from $z$ to $x_j$. Since $(\widetilde M, \widetilde F)$ is a compact (non-symmetric) metric space it follows form~\cite{myers1945arcs} there exists a rectifiable curve $\beta_{\infty}$ connecting $z$ to $x'$, that is a uniform limit of $\beta_{j}$ and whose length is not greater than $d_{\widetilde F}(x',z)$.  This implies $\beta_{\infty}(t)=\gamma_{x,v}(h-t),$ since $\gamma_{x,v}$ is the unique ${\widetilde F}$ distance minimizer from $x'$ to~$z$. 

Since $z\in \widetilde M^{int}$, there exists $R>0$ such that for every $j_k\in \N$ the curve $\beta_{j}(t)$ is a geodesic of $(\widetilde M, \stackrel{\leftarrow}{\widetilde F})$ if $t\in [0,R]$. Therefore, 
\begin{equation}
\label{eq:converg_of_initial_direct}
\dot{\beta}_j(0)\to -\eta \in S_{z}\widetilde M
\end{equation}
and the continuity of the exit time function $\tau_{exit}$ implies that there exists $J\in \N$ such that for every $j>J$ the curve $\beta_{j}$ is a geodesic of $(M, \stackrel{\leftarrow}{ F})$. Thus~\eqref{eq:local_dist_func} and~\eqref{eq:converg_of_initial_direct} contradict with the assumption that $\stackrel{\leftarrow}{\widetilde \exp}_{z}$ is a diffeomorphism near $-h\eta$. Therefore,~\eqref{eq:local_dist_func} cannot hold and we have that $d_{\widetilde F}(\cdot, z)$ and the local distance function $ \rho(\cdot,z)$ coincide near $x'$. 
Hence there exists a neighborhood $A\subset M^{int}$ of $x'$ in which we have
\begin{equation}
d_{F}(\cdot, z)=d_{\widetilde F}(\cdot, z)=\rho(\cdot,z),
\end{equation}
due to continuity of the exit time function.
\end{proof}

\begin{proof}[Proof of Lemma~\ref{Le:final_condition}]
We let $(z,\eta)$ be the exit point and direction of $\gamma_{x,v}$. We let $(\widetilde M, \widetilde F)$ be any compact Finsler manifold for which $z\in M^{int}$, and choose $s\in (0,\tau_{exit}(x,v))$ and denote $\ell=\ell(s):=\tau_{exit}(x,v)-s$. Then for $x':=(\gamma_{x,v}(s))$ the point $z$ is not a conjugate point along $\gamma_{x,v}$.  Since the conjugate distance function is lower continuous \cite[Section 12.1]{shen2001lectures} there exist neighborhoods $V\subset T_zM$ of $-\ell\eta$ and  $U\subset \widetilde M$ of $\gamma_{x,v}([s,\tau_{exit}(x,v)])$ such that for any $y\in V$ the shortest curve that is contained in $U$ and connects $z$ to the point $x(y):=\stackrel{\leftarrow}{\exp}_z(y)\in U$, is the geodesic $t\mapsto \;\stackrel{\leftarrow}{\exp}_z(ty), \: t\in [0,1]$.

We let $(h_j)_{j=1}^\infty, \: (\epsilon_j)_{j=1}^\infty\subset \R$ be such that $h_j, \epsilon_j>0$, $h_j,\epsilon_j\to 0$ as $j\to \infty$.  Denote $x_j:=\stackrel{\leftarrow} {\gamma}_{z,\xi_{h_j}}(\ell-\epsilon_j)$ , where $\xi_{h_j}:=\frac{-\eta+h_j\stackrel{\leftarrow} {\nu}_{in}}{\stackrel{\leftarrow} {F}(z,-\eta+h_j\stackrel{\leftarrow} {\nu}_{in})}$. Then $x_j\to x'$ as $j\to \infty$. We let $c_j:[0,\ell]\to M$ be a distance minimizing curve of $(M,\stackrel{\leftarrow} {F})$ from $z$ to $x_j$. Then we have 
\begin{equation}
\L(c_j)\longrightarrow \ell \hbox{ as } j\longrightarrow \infty.
\end{equation}
Due to~\cite{myers1945arcs} we can without loss of generality assume that curves $c_j$ converge uniformly to a rectifiable curve $c_\infty$, from $z$ to $x'$, that satisfies
\begin{equation}
\L(c_\infty)\leq \ell.
\end{equation} 
Then \cite[Proposition 11.3.1]{shen2001lectures} implies $c_\infty=\stackrel{\leftarrow} {\gamma}_{z,-\eta}|_{[0,\ell]}$, since otherwise there would exist a distance minimizing curve of $(M,\stackrel{\leftarrow} {F})$ from $z$ to $x$ that is not $C^1$-smooth at $x'$. Since $c_j \to\stackrel{\leftarrow} {\gamma}_{z,-\eta|_{[0,\ell]}}$ uniformly in $[0,\ell]$ there exists $J\in \N$ such that for all $j\geq J $ we have 
$( \stackrel{\leftarrow}{\exp}_z)^{-1}(x_j)\in V$ and the image of $c_j$ is contained in $U$. Thus after unit speed reparametrization of $c_j$ we have, $c_j(t)=\stackrel{\leftarrow} {\gamma}_{z,\xi_{h_j}}(t)$ for any $t\in [0,\ell-\epsilon_j]$. 

We recall that above we had $\ell=\tau_{exit}(x,v)-s$. The claim of this lemma follows using a diagonal argument for sequence $(\epsilon^j_i)_{i=1}^\infty,(h^j_i)_{i=1}^\infty\subset (0,1)$ which are chosen as above for $s_j\in (0,\tau_{exit}(x,v)),\: j \in \N$, such that $s_j \to 0$ when $j\to 0$.
\end{proof}

\subsection{Finsler structure} 

In this section, we prove that the data~\eqref{eq:data} determine the  set $\widehat{G}(M,F)$, (see~\eqref{eq:good_set}) and the  Finsler function $F$ on it. Again we deal separately with interior and boundary cases.

\begin{Lemma}
Let $x \in M^{int}$. The set $T_x M \cap \widehat{G}(M, F)$ contains an open non-empty set. Moreover  the data~\eqref{eq:data} determine the set $T_x M \cap \widehat{G}(M, F)$ and $F$ on it.
\end{Lemma}
\begin{proof}
We let  $z_x\in \p M$ be a closest boundary point to $x$. By Lemma~\ref{Le:boundary_vs_normal_cut_dist} the function $d_F(\cdot,z_x)$ is smooth at $x$. Moreover 
\begin{equation}
v:=-\dot{\stackrel{\leftarrow}{\gamma}}_{z_x,\stackrel{\leftarrow}{\nu}_{in}(z_x)}(d_F(x,z_x)) \in S_xM \cap \widehat{G}(M,F).
\end{equation}
Thus the set $S_xM \cap \widehat{G}(M,F)$ is not empty.  By Corollary~\ref{cor:39} there exists a neighborhood $U\subset S_xM$ of $v$ that is contained in $\widehat{G}(M,F)$. 

Next, we prove the latter claim. We let $z \in \p M$ be such that the function $d_F(\cdot,z)$ is smooth at $x$. 
We let $v \in \stackrel{\leftarrow}{ S_x M}$. Then
\begin{equation}
\label{eq:F_star }
d(d_{\stackrel{\leftarrow}{F}}(z,\cdot))\bigg|_x = \stackrel{\leftarrow}{g_{v}}(v,\cdot)=  \stackrel{\leftarrow}{\ell_x}(v) \quad \hbox{if and only if} \quad  \gamma_{x,-v}(d_F(x,z))=z,
\end{equation}
where $g_{y}(\cdot,\cdot)$ is the hessian of $\frac 12 F^2(x,y)$ with respect to $y$ variables and $\stackrel{\leftarrow}{\ell_x}$ is the Legendre transform of Finsler function $\stackrel{\leftarrow}{F}$ at $x$. The property~\eqref{eq:F_star } implies that the set
\begin{equation}
\label{eq:F_star_2}
A(x):= \{d(d_{\stackrel{\leftarrow}{F}}(z,\cdot))\bigg|_x: z \in \p M, \: d_{\stackrel{\leftarrow}{F}}(z,\cdot) \hbox{ is $C^\infty$ at } x\} 
\end{equation}
satisfies 
\begin{equation}
A(x)=\:\stackrel{\leftarrow}{\ell_x}(\stackrel{\leftarrow}{ S_x M} \cap (-\widehat{G}(M,F))).
\end{equation}
Since the Legendre transform is an isometry the dual map $\bigg(\!\!\stackrel{\leftarrow}{F}\!\!\bigg)^\ast$ is constant $1$ on $A(x)$.  As the function $\bigg(\!\!\stackrel{\leftarrow}{F}\!\!\bigg)^\ast$ is positively homogeneous of order $1$ we have determined $\bigg(\!\!\stackrel{\leftarrow}{F}\!\!\bigg)^\ast$ on $\R_+A(x):=\{r p \in T^\ast_xM: p \in A(x), \: r>0\}$. Recall that the components of  the Legendre satisfy
\begin{equation}
\label{eq:recover_of_Legendre}
\bigg((\stackrel{\leftarrow}{\ell_x})^{-1}(p))\bigg)_j=\frac{1}{2}\bigg(\frac{\p}{\p p^i}\frac{\p}{\p p^j}\bigg(\bigg(\!\!\stackrel{\leftarrow}{F}\!\!\bigg)^\ast\bigg)^2(x,p)\bigg)p_i  
\end{equation}
for all $p \in T_x^\ast M.$ Since $\bigg(\!\!\stackrel{\leftarrow}{F}\!\!\bigg)^\ast$ is recovered on $\R_+A(x)$ the equation~\eqref{eq:recover_of_Legendre} determines $(\stackrel{\leftarrow}{\ell_x})^{-1}$ on $\R_+A(x)$. Therefore,
\begin{equation}
\begin{split}
&T_x M \cap (-\widehat{G}(M,F))=(\stackrel{\leftarrow}{\ell_x})^{-1}(\R_+A(x)) \hbox{ and } 
\\ 
&\stackrel{\leftarrow}{F}\; =\bigg(\bigg(\!\!\stackrel{\leftarrow}{F}\!\!\bigg)^\ast\bigg)^\ast \hbox{ on } T_x M \cap(- \widehat{G}(M,F)).
\end{split}
\end{equation}
Finally,
\begin{equation}
\begin{split}
&F(x,y)=\;\stackrel{\leftarrow}{F}(x,-y).
\end{split}
\end{equation}
This concludes the proof.
\end{proof}

\begin{Lemma}
\label{Le:recovery_of_F_in_interior}
Let $x \in M^{int}_1$. Then 
\begin{equation}
\widehat{G}(M_1, F_1)\cap T_xM_1=\widehat{G}(M_1,\Psi^\ast F_2)\cap T_xM_1
\end{equation}
and
\begin{equation}
F_1(y)=F_2(\Psi_\ast y), \quad y \in (\widehat G(M_1, F_1)\cap T_xM_1).
\end{equation}
\end{Lemma}
\begin{proof}
The mapping $\Psi$ is a diffeomorphism that satisfies
\begin{equation}
\label{eq:connection_of_distat_func}
d_{F_2}(\Psi(\cdot), \phi(\cdot))|_{M_1 \times \p M_1}=d_{F_1}(\cdot, \cdot)|_{M_1 \times \p M_1}.
\end{equation} 
Thus for any $z \in \p M_1$ the function $d_{F_2}(\Psi(\cdot), \phi(z))$ is smooth at $\Psi(x)$ if and only if $d_{F_1}(\cdot, z)$ is smooth at $x$. Therefore, the claims follow by applying the differential of $M_1$ to the both sides of~\eqref{eq:connection_of_distat_func} and using~\eqref{eq:F_star }--\eqref{eq:recover_of_Legendre}.
\end{proof}

\medskip 

Next, we consider the boundary case.

\begin{Lemma}
Let $x \in \p M$.
Then data~\eqref{eq:data} determine $F$ on 
\begin{equation}
\p_{out}TM \cap T_xM=\{y \in T_xM: g_{\nu_{in}}(\nu_{in},y)<0\}.
\end{equation}
\end{Lemma}

\begin{proof}
We let   $y \in T_xM \setminus \{0\}$ be an outward pointing vector that is not tangential to the boundary. We let $b>a\geq 0$ and choose any smooth curve $c\colon [a,b] \to M$ such that
\begin{equation}
c((a,b)) \subset M^{int}, \quad c(b)=x, \: \dot{c}(b)=y. 
\end{equation}
Recall that with respect to the geodesic coordinates at $x$ we have $d_F(x,c(t))=F(\exp^{-1}_x(c(t)))$. Since $F$ is continuous we have
\begin{equation}
\label{eq:F_on_boundary}
\lim_{t \to b}\frac{d_F(x,c(t))}{b-t}
=F(\dot c(b))=F(y).
\end{equation}
Since  $y \in T_xM \setminus \{0\}$ was an arbitrary outward pointing vector the data~\eqref{eq:data} and~\eqref{eq:F_on_boundary} determine $F$ on the set $\p_{out}TM \cap T_xM.$
\end{proof}

\begin{Lemma}
\label{Le:recovery_of_F_on_boundary}
For any $(x,y) \in \p_{out}TM_1$ it holds that 
$$
F_1((x,y))=F_2(\Psi_\ast (x,y)).
$$
\end{Lemma}
\begin{proof}
The claim follows from~\eqref{eq:connection_of_distat_func} and~\eqref{eq:F_on_boundary}.
\end{proof}

Now we are ready to give a proof of Theorem~\ref{Th:smooth}.

\begin{proof}[Proof of Theorem~\ref{Th:smooth}]
By theorems~\ref{Th:topology} and~\ref{Th:diffeo} the  map $\Psi\colon M_1 \to M_2$ is a diffeomorphism, and
the pullback $\Psi^*F_2$ of $F_2$ gives a Finsler function on $M_1$. By Lemmas~\ref{Le:recovery_of_F_in_interior} and~\ref{Le:recovery_of_F_on_boundary} we have proved that  $\overline{G(M_1,F_1)}$ and $\overline{ G(M_1,\Psi^*F_2)}$ coincide and in this set $F_1=\Psi^*F_2$.

\medskip 
We still have to show that the data~\eqref{eq:data} are not sufficient to guarantee that $F_1$ and $F_2$ coincide in $TM_1^{int}\setminus \overline{G(M_1,F)}$. We denote a manifold $M_1$ by $M$ and a Finsler function $F_1$ by $F$. If  $TM^{int}\setminus \overline{G(M,F)}$ is not empty, we choose $(x_0,v_0) \in TM^{int}\setminus \overline{G(M,F)}, \:F(v_0)=1$ and a neighborhood $V \subset TM^{int}\setminus \overline{G(M,F)}$ of $(x_0,v_0)$ such that
\begin{equation}
\label{eq:V_is_far_from_pM}
\dist_{F}(\pi(V), \p M)>0.
\end{equation}
We denote the radial projection of $V$ to the unit sphere bundle of $(M,F)$ by~$W$.
We let $\alpha \in C_0^ \infty(W)$ be non-negative and  define a function
\begin{equation}
\label{eq:distorted_Fins}
H\colon \R \times TM \to \R, \quad H(s,y)=\bigg(1+s \alpha\bigg(\frac{y}{F(y)}\bigg)\bigg)F(y).
\end{equation}
We show that there exists $\epsilon >0$ such that for any $s\in (-\epsilon,\epsilon)$ the function $H(s,\cdot)\colon TM \to \R$ is a Finsler function. 

\medskip
Since $\alpha$ is compactly supported it holds,  for $|s|$ small enough, that  $H(s,\cdot)$ is non-negative, continuous and $H(s,y)=0$ if and only if $y=0$. Moreover $H(s,\cdot)$ is smooth outside the zero section of $M$.  Clearly also the scaling property $H(s,ty)=tH(s,y), \: t>0$ is valid. 

We let $(x,y)$ be a smooth coordinate system of $TM$ near $(x_0,v_0)$. To prove that $H(s,\cdot)$ is a Finsler function, we  have to show that for every $(x,y) \in TM \setminus \{0\}$ the Hessian 
\begin{equation}
\frac{1}{2}\frac{\p}{\p y^i}\frac{\p}{\p y^j}H^2(s,(x,y))=\frac{1}{2} \frac{\p}{\p y^i}\frac{\p}{\p y^j}\bigg[\bigg(1+s \alpha\bigg(\frac{y}{F(y)}\bigg)\bigg)^2\bigg(F(y)\bigg)^2\bigg]
\end{equation}
is symmetric and positive definite. Since $H^2(s,(x,\cdot))\colon T_xM \to \R$ is smooth outside $0$ it follows that the Hessian is symmetric,
and $\alpha \in C^\infty_0(W)$ implies
\begin{equation}
\begin{split}
\frac{1}{2} \frac{\p}{\p y^i}\frac{\p}{\p y^j}H^2(s,(x,y))&
 = g_{ij}(x,y) +\mathcal{O}(s)
 \end{split}
\end{equation}
where $g_{ij}$ is the Hessian of $\frac{1}{2}F^{2}$. Therefore, for $|s|$ small enough $H(s,\cdot)\colon TM \to \R$ is a Finsler function. 

\medskip
We let $\epsilon >0$ be so small that $H(s,\cdot)$ is a Finsler function for $s \in (0,\epsilon)$. We prove that for any $x \in M$ and $z \in \p M$
\begin{equation}
\label{eq:boundary_data_round_metric_and_distorted_metric}
d_{H(s,\cdot)}(x,z)= d_{F}(x,z).
\end{equation}
This implies that the boundary distance data of $F$ and $H(s,\cdot)$ coincide. 

If $c\colon [0,1] \to M$ is any piecewise $C^1$-smooth curve,~\eqref{eq:distorted_Fins} implies
\begin{equation}
\label{eq:lenghts_of_round_metric_and_distorted_metric}
\mathcal{L}_{F}(c)\leq \mathcal{L}_{H(s,\cdot)}(c).
\end{equation}
We let $x \in M$ and $z \in \p M$. Since $M$ is compact there exists a $F$-distance minimizing curve $c\colon [0,d_{F}(x,z)] \to M$ from $x$ to $z$. We let $I, J \subset [0,d_{F}(x,z)] $ be a partition $[0,d_{F}(x,z)]$ such that
\begin{equation}
c(t)\in M^{int} \quad \hbox{ if and only if } \quad t\in I.
\end{equation}
Then $I$ is open in $[0,d_{F}(x,z)]$ and $J$ is closed. On set $I$ the curve $c$ is a union of distance minimizing geodesic segments of $F$ which have end points in $\p M$. Thus for any  $t \in I$ we have $\dot{c}(t) \in G(M,F)$. This and~\eqref{eq:V_is_far_from_pM} imply 
\begin{equation}
d_{F}(x,z)=\mathcal{L}_{F}(c)= \mathcal{L}_{H(s,\cdot)}(c),
\end{equation}
and the equation~\eqref{eq:boundary_data_round_metric_and_distorted_metric} follows from~\eqref{eq:lenghts_of_round_metric_and_distorted_metric}.
\end{proof}

\section{Completion of the proof of Lemma~\ref{Le:prop_boundary_cut}}
\label{Se:Appendix2}

Most of Lemma~\ref{Le:prop_boundary_cut} was proven after its statement, but the final claim was deferred here.
The proof of the missing claim is presented at the end of this section, once we have introduced all the tools we use.

In this section, we denote by $(N,F)$ a compact, connected smooth Finsler manifold without boundary. We present the second variation formula in the case
when the variation curves start from a smooth submanifold $S$ of
$N$. We introduce the concept of a focal distance and connect it to
the degeneracy of the normal exponential map $\exp^\perp$ of surface
$S$.

\medskip
We define a pullback vector bundle $\pi^\ast TN$ over $TN\setminus \{0\}$ such that for every $(x,y)\in TN\setminus \{0\}$ the corresponding fiber is $T_xN$. Notice that  $\pi^\ast TN$ is then defined by the equation
\begin{equation}
\pi^\ast TN=\{((x,y),(x,y'))\in (T_xN\setminus \{0\})\times T_xN: x \in N\}.
\end{equation}
We let $(x,y)$ be a local coordinates for $TN$. We define a  local frame  $(\p_i)_{i=1}^n$ for  $\pi^\ast TN$ by
\begin{equation}
\label{eq:pullback_bundle_frame}
\p_i|_{(x,y)}:=\bigg((x,y),\frac{\p}{\p x^i}\bigg).
\end{equation}
and a local co-vector field on $TN$ by
\begin{equation}
\delta y^i:= dy^i+N^i_jdx^j, \quad N^i_j(x,y):=\frac{\p}{\p y^j} G^i(x,y).
\end{equation}
Above the functions $G^i$ are the geodesic coefficients of $F$ in coordinates $(x,y)$ (see~\eqref{eq:geodesic_coef_2} in~\ref{Se:Appendix1}).
Notice that $\frac{\p}{\p y^i}$ is a dual vector to $\delta y^i$ and a dual vector  $\frac{\delta }{\delta x^i}$ to $dx^i$ is given by
\begin{equation}
\label{eq:horisontal_part}
\frac{\delta }{\delta x^i}:=\frac{\p }{\p x^i}-N^j_i\frac{\p }{\p y^j}.
\end{equation}
Therefore, vectors $dx^i$ and $\delta y^j$ are linearly independent for all $i,j \in \{1, \ldots, n\}$ and it holds that  
\begin{equation}
\label{eq:horizontal_hertical_decomposition}
T^\ast(TN\setminus \{0\})=\Span\{dx^i\}\oplus \Span\{\delta y^i\}=:\mathcal{H}^\ast(TN) \oplus\mathcal{V}^\ast(TN).
\end{equation}

We relate $\pi^\ast TN$ locally to $\mathcal{H}^\ast(TN)$ and to $\mathcal{V}^\ast(TN)$ by mappings
\begin{equation}
\label{eq:relation_of_pullpack_H_V}
\p_i \mapsto dx^i \hbox{ and } \p_i \mapsto \delta y^i, \: i\in \{1,\ldots, n\}.
\end{equation}

We denote the collection of smooth sections of  $\pi^\ast TN$ by $\mathcal{S}( \pi^\ast TN)$. The \textit{Chern connection} is defined on $ \pi^\ast TN$ by
\begin{equation}
\label{eq:Chern_Connect}
\nabla\colon \mathcal{T}(TN)\times \mathcal{S}( \pi^\ast TN)\to \mathcal{S}(\pi^\ast TN), \quad \nabla_XU=\bigg\{dU^i(X)+U^j\omega^i_j(X)\bigg\}\p_i,
\end{equation}
where $\mathcal{T}(TN)$ is the collection of all smooth vector fields on $TN\setminus \{0\}$ and the connection one forms $\omega^i_j$ on $TN\setminus \{0\}$ are given by
\begin{equation}
\label{eq:Chern_forms}
\omega^i_j(x,y):=\Gamma^i_{jk}(x,y)dx^k,
\end{equation}
and functions $\Gamma^i_{jk}(x,y)$ are defined by  \cite[equation (5.25)]{shen2001lectures}. They satisfy
\begin{equation}
\label{eq:props_of_Christof}
y^k\Gamma^i_{jk}(x,y)=N^i_j(x,y) \quad \hbox{ and } \quad  \Gamma^i_{jk}= \Gamma^i_{kj}. 
\end{equation}
See \cite[equations (5.24) and (5.25)]{shen2001lectures}.


Notice that any vector field  $X$ on $TN$, that is locally given by 
\begin{equation}
X=X^i_x\frac{\delta}{\delta x^i}+X^i_y\frac{\p}{\p y^i}, \quad X^i_x,X^i_y\in C^\infty(TN)
\end{equation}
defines a section $\widetilde X \in \mathcal{S}( \pi^\ast TN)$ by
\begin{equation}
\widetilde X(x,y) =X^i_x(x,y)\p_i.
\end{equation}

\begin{Lemma}
Let $X,Y,Z$ be vector fields on $TN$. Then
\begin{equation}
\label{eq:Chern_prop_1}
 \nabla_X\widetilde Y- \nabla_Y\widetilde X=\widetilde{[X,Y]}.
\end{equation}
\end{Lemma}

\begin{proof}
Equation~\eqref{eq:Chern_prop_1} follows from the definition of the Chern connection and~\eqref{eq:props_of_Christof}.
\end{proof}

The fundamental tensor $g$ on $\pi^\ast TN$ is defined by
\begin{equation}
g(U,V):=g_{ij}(x,y)U^i(x,y)V^i(x,y), \quad U,V \in  \mathcal{S}( \pi^\ast TN), \: (x,y) \in TN,
\end{equation}
where $g_{ij}(x,y)=g_y(\frac{\p}{\p x^i},\frac{\p}{\p x^j})$.

Recall that if $X=X^i_x\frac{\delta}{\delta x^i}+X^i_y\frac{\p}{\p y^i}$ is in $T(TN)$, then $D\pi X=X^i_x\frac{\p}{\p x^i} \in TN$. 

\begin{Lemma}
\label{Le:Chern_prop_2}
Let $X, Y, Z$ be vector fields on $TN$.  Then
\begin{equation}
\label{eq:Chern_prop_2}
Yg(\widetilde X, \widetilde Z)\bigg|_{D\pi X}=\bigg[g(\nabla _Y\widetilde X,\widetilde Z)+g(\widetilde X,\nabla _Y\widetilde Z)\bigg]\bigg|_{D\pi X}.
\end{equation}
\end{Lemma}

\begin{proof}
The proof is a direct evaluation in coordinates, using that $g_{ij}$
is homogeneous of order zero with respect to directional variables. It is important to
evaluate $Y g(\widetilde X, \widetilde Z)$ at $D\pi X$ as
for an arbitrary direction,~\eqref{eq:Chern_prop_2} does not
hold, since the Cartan tensor does not vanish identically.
%
\end{proof}

If $ V(x)=V^i(x)\frac{\p}{\p x^i}$ is a vector field on $N$ then $\widehat V(x,y):=V^i(x)\frac{\delta}{\delta x^i}$ is a horizontal vector field on $TN$. We call $\widehat V$ a horizontal lift of $ V$.  We define a covariant derivative $D_t$ of smooth vector field $V$ on geodesic $\gamma$ as
\begin{equation}
\label{eq:covariant_der}
D_tV(t):= \bigg\{\dot{V}^i(t)+V^j(t)N_j^i(\dot{\gamma}(t))\bigg\}\frac{\p}{\p x^i}\bigg|_{\gamma(t)}.
\end{equation}
In the next lemma we relate the covariant derivative to the Chern connection.

\begin{Lemma}
Let $t \mapsto c(t)$ be a geodesic on $(N,F)$ and $V$ be a smooth vector field on $c$ that is extendible.  Write $V(x)=V^i(x)\frac{\p}{\p x^i}$. Then $\widetilde V(x,y):=V^i(x)\p_i|_{(x,y)}$ is a smooth section of $ \mathcal{S}( \pi^\ast TN)$ and
\begin{equation}
\label{eq:Chern_prop_3}
\widetilde{D_t V}=\nabla_{\widehat{\dot c}} \widetilde V.
\end{equation}
\end{Lemma}

\begin{proof}
The claim can be verified in local coordinates using~\eqref{eq:props_of_Christof}.
\end{proof}


We now consider variations of a geodesic $\gamma\colon[0,h]\to N$ normal to
a hypersurface~$S$ so that one endpoint stays on~$S$ and the other one
is fixed. We denote the starting point of $\gamma$ by $z_0\in S$.
For a smooth curve $\sigma$ on $S$ we assume that a variation $\Gamma(s,t)$ satisfies
$\Gamma(s,0)=\sigma(s)$,
$\Gamma(s,h)=\gamma(h)$, and
$\Gamma(0,t)=\gamma(t)$
for all values of the time $t\in[0,h]$ and the variation parameter $s$
near zero.
The variation field $J(t):=\frac{\p}{\p s} \Gamma(s,t)|_{s=0}$ is a
vector field along $\gamma$ and satisfies the boundary conditions
$J(0)=\dot{\sigma}(0)$ and $J(h)=0$.
We additionally assume that the variation is normal:
$g_{\dot{\gamma}}(\dot{\gamma}, J)\equiv 0$.

The second variation formula (see \cite[Chapter 10]{shen2001lectures})
and equations~\eqref{eq:Chern_prop_2}--\eqref{eq:Chern_prop_3} imply that
\begin{equation}
\frac{\p^2}{\p s^2}\mathcal{L}(\Gamma(s,\cdot))\bigg|_{s=0}=\int_0^hg_{\dot
\gamma}(D_t^2 J(t)-\textbf{R}_{\dot{\gamma}}( J(t)), J(t))dt+
g(\nabla_{\widehat J} \widetilde\nu-\nabla_{\widehat \nu} \widetilde
J, \widetilde{J})\bigg|_{\nu},
\end{equation}
where $\textbf{R}_{( \cdot)}( \cdot)$ is the \textit{Riemannian
curvature tensor} (see \cite[Chapter 6]{shen2001lectures}).
If $J$ is a Jacobi field, the expression simplifies to
\begin{equation}
\frac{\p^2}{\p s^2}\mathcal{L}(\Gamma(s,\cdot))\bigg|_{s=0}=
g(\nabla_{\widehat J} \widetilde\nu-\nabla_{\widehat \nu} \widetilde
J, \widetilde{J})\bigg|_{\nu}.
\end{equation}
To discuss geodesic variations of $\gamma$, we consider normal Jacobi
fields $J$ along $\gamma$ that satisfy
\begin{equation}
\label{eq:Transverse_J_fields}
J(0)\in TS \hbox{ and } \bigg(\nabla_{\widehat J}
\widetilde\nu-\nabla_{\widehat \nu} \widetilde J\bigg)\bigg|_{\nu}=0.
\end{equation}
We let $\mathcal{J}$ be the collection of all normal Jacobi fields on
$\gamma$ that satisfy~\eqref{eq:Transverse_J_fields} and $\mathcal
J_0=\{J\in\mathcal J:J(h)=0\}$.
Following~\cite{Chavel}, we call $\mathcal{J}$ the space of
\textit{transverse Jacobi fields}.

\begin{Lemma}
\label{Le:char_of_trans_jacobi}
A vector field $J$ on $\gamma$ is a transverse Jacobi field if and only if
\begin{equation}
\label{eq:char_of_trans_jacobi}
J(t)=D\exp^\perp\bigg|_{(z_0,t)} t \eta \quad \hbox{ for some } \eta =J(0)\in T_{z_0}S. 
\end{equation}
\end{Lemma}
\begin{proof}
We let $\epsilon>0$ and  $U\subset S$ be a neighborhood of $z_0$ be such that  the normal exponential map $\exp^\perp\colon U \times (-\epsilon,\epsilon)\to N$ is a diffeomorphism onto its image. We define  a unit length vector field $W$, that is orthogonal to $S$, by
\begin{equation}
W(x):=\frac{\p}{\p t}\exp^\perp(z,t)=\dot{\gamma}_{z,\nu(z)}(t), \quad (z,t)\in U \times (-\epsilon,\epsilon),
\end{equation}
where $x=\exp^\perp(z,t)$. 

For any $z\in U$ the geodesic $\gamma_{z,\nu(z)}$ of $F$ is also a geodesic of the local Riemannian metric 
\begin{equation}
g_W(x):=\hbox{Hess}_y\bigg(\frac 12 F(x,y)^2\bigg)\bigg|_{y=W(x)},
\end{equation}
that is normal to $S$. This implies that the normal exponential maps of~$F$ and $g_W$ coincide in $U \times (-\epsilon,\epsilon)$ and moreover $D_t =D^W_t$, where $D_t$ is the covariant derivative of~$F$ (given in~\eqref{eq:covariant_der}) and $D^W_t$ is the covariant derivative of Riemannian metric $g_W$ on $\gamma_{z,\nu(z)}$. 

%

\medskip
Now we are ready to prove the claim of this lemma. We let $\sigma(s)\in S$ be a smooth curve with initial conditions $\sigma(0)=z_0$ and $\dot \sigma(0)=\eta=J(0)$.  Define $\Gamma(s,t)=\exp^\perp(\sigma(s),t)$. Then $\Gamma(0,t)=\gamma(t)$ and all the variation curves $t \mapsto\Gamma(s,t)$ are geodesics. Therefore, the variation field
\begin{equation}
V(t):=\frac{\p}{\p s}\Gamma(s,t)\bigg|_{s=0}
\end{equation}
with $t  \in(-\epsilon,\epsilon)$ 
is a Jacobi field of $F$ that satisfies 
\begin{equation}
V(0)=\frac{\p}{\p s}\Gamma(s,t)\bigg|_{s=t=0}=D\exp^\perp\bigg|_{(z_0,0)}\dot\sigma(0)=\eta.
\end{equation}
Therefore, it suffices to show that $D_tJ(0)=D_tV(0)$. We note that since $g_W$ is a Riemannian metric the following symmetry holds true,
\begin{equation}
\label{eq:symmetry_formula}
D_t \frac{\p}{\p s} =D_s\frac{\p}{\p t},
\end{equation}
along any transverse curve $s\mapsto \Gamma(s,t)$. Above $D_s$ is a covariant derivative  of $g_W$ along transverse curve $\Gamma(\cdot,t)$. We also assume that $ \dot{\sigma}$ can be extended to a smooth vector field $Y$ on $N$. Then the equation~\eqref{eq:symmetry_formula} implies
\begin{equation}
\begin{split}
D_tV(0)=
\overline\nabla_{Y}W\bigg|_{z_0}=\overline\nabla_{\eta}W,
\end{split}
\end{equation}
where $\overline{\nabla}$ is the Riemannian connection of~$g_W$. To end the proof we still have to show the first equation of the following
\begin{equation}
\widetilde{\overline\nabla_{\eta}W}=\nabla_{\widehat \eta}\widetilde W=\nabla_{\widehat W}\widetilde \eta=\widetilde{ D_t J(0)},
\end{equation}
where $\nabla$ is the Chern connection of~$F$. The proof of this claim is a direct computations in local coordinates.
\end{proof}

We obtain the following lemma as a direct consequence.

\begin{Lemma}
\label{lma:dimJ=n-1}
Set $ \mathcal{J}$ is a real vector space of dimension $n-1$. 
\end{Lemma}

\begin{proof}
The claim follows since the dimension of $S$ is $n-1$ and the operator given by~\eqref{eq:char_of_trans_jacobi} is linear in $T_{z_0}S$ and onto at $t=0$.
\end{proof}
\medskip

Similarly to the spaces $\mathcal J$ and $\mathcal J_0$ of Jacobi
fields defined above, we denote by~$\mathcal V$ the collection of
piecewise smooth normal vector fields along $\gamma$
satisfying~\eqref{eq:Transverse_J_fields} and by~$\mathcal V_0$ the
subspace vanishing at $\gamma(h)$.
On $\mathcal{V}_0$ we define the index form 
\begin{equation}
\label{eq:index_form}
\begin{array}{rl}
I(V,W):=&\int_0^hg_{\dot \gamma}(D_t  V(t),D_t  W(t))-g_{\dot \gamma}(\textbf{R}_{\dot{\gamma}}( V(t)),  W(t))dt
\\
-& g(\nabla _{\widehat W}\widetilde V\bigg|_{\dot \gamma(0)},\widetilde\nu ).
\end{array}
\end{equation}


\begin{Lemma}
\label{lma:index-form-bilinear}
The index form $I$ on $\mathcal{V}_0$ is a symmetric bilinear form.
\end{Lemma}

\begin{proof}
Clearly, $I$ is bilinear. It is proven in \cite[Section 8.1]{shen2001lectures} that for all $x\in N$ and $y,v,w \in T_xN$, it holds that
\begin{equation}
g_{y}(\textbf{R}_{y}(v),w)=g_{y}(v,\textbf{R}_{y}(w)) .
\end{equation}
Since $V, W$ are normal to $\dot{\gamma}$, the equation
\begin{equation}
 g(\nabla _{\widehat W}\widetilde V\bigg|_{\dot \gamma(0)},\widetilde\nu )
 =g(\nabla _{\widehat V}\widetilde W \bigg|_{\dot \gamma(0)},\widetilde\nu)
\end{equation}
follows from~\eqref{eq:Chern_prop_2} and the symmetry of the second fundamental form (see \cite[Section 14.4]{shen2001lectures}).
\end{proof}

\begin{Lemma}
\label{Le:Tangential_variations}
Assume  that $\gamma$ is not self-intersecting on $[0,h]$. We let $ V\in \mathcal{V}$. There exists $\delta>0$ and a variation $\Gamma\colon(\delta,\delta) \times [0,h]\to N$  of $\gamma$ whose variation field $\frac{\p}{\p s}\Gamma(s,t)|_{s=0}$ is $V(t)$ and $\Gamma(s,0)$ is a smooth curve on $S$. Moreover if $t_1, \ldots, t_k \in [0,h]$ are the points where $ V$ is not smooth then  $\Gamma\colon (\delta,\delta) \times (t_i,t_{i+1}) \to N$ smooth. 
\end{Lemma}
\begin{proof}
\medskip
We let $W$ be a smooth vector field that is an extension of $\dot{\gamma}(t)$ in a neighborhood of $\gamma([0,h])$. Using the Fermi coordinates of $S$, with respect to the local Riemannian metric $g_W$,
we can construct a Riemannian metric $\widetilde g$ in some neighborhood of $\gamma([0,h])$ such that $S$ is a geodesic submanifold of $\widetilde g$ and $\gamma$ is a geodesic of $\widetilde g$ that is $\widetilde g$-normal to $S$. Then we use the following variation to prove the claim of this lemma.

\medskip
We let $\delta>0$ and define a variation of $\gamma$ with
\begin{equation}
\label{eq:Variation_of_V}
\Gamma(s,t)=\exp_{\widetilde g}(\gamma(t),sV(t));\: t\in [0,h],  s\in (-\delta,\delta),
\end{equation}
where $\exp_{\widetilde g}$ is the exponential map of metric tensor $\widetilde g$. Since $S$ is a geodesic submanifold with respect to $\widetilde g$ we have that 
\begin{equation}
\Gamma(s,0)=\exp_{\widetilde g}(\gamma(0),sV(0)) \in S,  \quad  s\in (-\delta,\delta).
\end{equation}
Moreover
\begin{equation}
\frac{\p}{\p s} \Gamma(s,t)\bigg|_{s=0}=D((\exp_{\widetilde g})_{\gamma(t)})\bigg|_0V(t)=V(t).
\end{equation}
The claim is proven.
\end{proof}

For a given vector field $V \in \mathcal{V}_0$ we call the variation of  $\gamma(t)$ given by $\eqref{eq:Variation_of_V}$ the variation related to $V$. 

\medskip
\begin{Definition}
We say that $\gamma(h)$ is a focal point of $S$ if the set $\mathcal{J}_0$ contains a non-zero Jacobi field.
\end{Definition}

\begin{Lemma}
\label{Le:Char_of_trans_Jacobi}
The point $\gamma(h)$ is a focal point of $S$ if and only if $D\exp^\perp$ is singular at $(z_0,h)$.
\end{Lemma}

\begin{proof}
The claim follows from Lemma~\ref{Le:char_of_trans_jacobi}.
%
\end{proof}

We define the quantities $\tau_S(z_0)$ and $\tau_f(z_0)$ as
\begin{equation}
\label{eq:tau_S}
\tau_S(z_0)=\sup\{t>0:t=d_F(z_0,\gamma_{z_0,\nu}(t))=d_F(S,\gamma_{z_0,\nu}(t))\}
\end{equation}
and
\begin{equation}
\label{eq:tau_f}
\tau_f(z_0)=\inf\{t>0: \gamma(t) \hbox{ is a focal point to $S$}\}.
\end{equation}
The function $\tau_S$ above agrees with $\tau_{\p M}$ given in Definition~\ref{De:boundary_cut_points} when $S=\partial M$.
Our final goal is to show that $\tau_S(z_0) \leq \tau_f(z_0)$.
To check the inequality, we still have to state one auxiliary result:

\begin{Lemma}
\label{lma:410}
If $\tau_f(z_0)> h$, then Index form $I$ is positive definite on $\mathcal V_0$. If $\tau_f(z_0)= h$, then $I$ is positive semidefinite on $\mathcal V_0$ and $I(V,V)=0$ if and only if $V\in \mathcal{J}_0$.
\end{Lemma}

\begin{proof} The claim follows from standard properties of the index form $I(\,\cdotp,\,\cdotp)$, see the proof of  \cite[Theorem II.5.4]{Chavel} for details. 
\end{proof}

\begin{Lemma}
\label{lma:negative-index-form}
Suppose that $\tau_f(z_0)< h$. Then there exists $W \in \mathcal{V}_0$ such that
\begin{equation}
I(W,W)<0.
\end{equation}
Moreover,
\begin{equation}
\label{eq:bnd_cut_poits_ocur_bf_focal}
\tau_S(z_0) \leq \tau_f(z_0).
\end{equation}
\end{Lemma}

\begin{proof}
Denote $\tau_f(z_0):=t_0 <h$. Choose a non-zero $J\in \mathcal{J}$ that vanishes at $t_0$. Define
\begin{equation}
V(t)
=
\begin{cases}
J(t),& t \leq t_0
\\
0,& t \in [t_0,h).
\end{cases}
\end{equation}
By Lemma~\ref{lma:410} it holds that $I(V,V)=0$. Since $D_t J(t_0)\neq 0$ there exists a non-zero smooth vector field $X \in \mathcal{V}_0 $ on $\gamma$ that satisfies
\begin{equation}
\hbox{supp}X\subset (0,h) \hbox{ and } X(t_0)=-D_tJ(t_0).
\end{equation}
Therefore if $\epsilon >0$ is small enough
 $I(V+\epsilon X, V+\epsilon X)$ is negative.

\medskip

Finally, we prove~\eqref{eq:bnd_cut_poits_ocur_bf_focal}. We denote $W:=V+\epsilon X$. We let $\Gamma(s,t)$ be the variation of $\gamma(t)$ that is related to $W$. Since $\gamma$ is a geodesic we have
\begin{equation}
\frac{d}{ds}\mathcal{L}(\Gamma(s,\cdot))=0 \hbox{ and } \frac{d^2}{ds^2}\mathcal{L}(\Gamma(s,\cdot))=I(W,W)<0.
\end{equation}
Therefore, $\gamma$ cannot minimize the length from $S$ to $\gamma(h)$. Thus the inequality~\eqref{eq:bnd_cut_poits_ocur_bf_focal} is valid.
\end{proof}

Let us now summarize the proof of the remaining claim.

\begin{proof}[Proof of the final statement of Lemma~\ref{Le:prop_boundary_cut}]
The claim remaining to be proven is that for any $t\in[0,t_0)$ the normal exponential map $\exp^\perp$ is non-singular at $(z,t)$.
The time $t_0=\tau_{\partial M}(z)$ is the boundary cut distance at $z$.
We are only interested in the hypersurface $S=\partial M$, so our cut distance function to $S$ equals that of Definition~\ref{De:boundary_cut_points}.

By Lemma~\ref{Le:Char_of_trans_Jacobi} we need to show that the point $\exp^\perp(z,t)$ is not a focal point.
The closest focal point to $z$ is at distance $\tau_f(z)$ from the boundary, and so by Lemma~\ref{lma:negative-index-form} there are no focal points before the time $\tau_S(z)=\tau_{\partial M}(z)$.
Therefore there are no focal points before the boundary cut distance, as claimed.
\end{proof}


\subsection*{Acknowledgements}

MVdH was supported by the Simons Foundation under the MATH + X program, the National Science Foundation under grant DMS-1815143, and by the members of the Geo-Mathematical Imaging Group at Rice
University.
JI was supported by the Academy of Finland (decisions
295853, 332890, and 336254).
Much of the work was completed during JI's visits to Rice
University, and he is grateful for hospitality and support.
ML was supported by Academy of Finland (decisions 284715 and 303754).
TS was supported by the Simons Foundation under the MATH + X program.
Part of this work was carried out during TS's visit to University of Helsinki, and he is grateful for hospitality and support.
We thank Antti Kykkänen and Jonne Nyberg for discussions and the anonymous referees for their valuable comments.

\appendix

\section{Basics of compact Finsler manifolds}

\label{Se:Appendix1}

In this appendix, we summarize some basic theory of compact Finsler manifolds. This section is intended for the readers having background in imaging methods and elasticity. We
follow the notation of~\cite{shen2001lectures} and use it as a main
reference. The main goal is to prove that if $x \in
M^{int}$ and $z_x \in \p M$ is a closest boundary point to $x$, that
is the minimizer of $d_F(x,\cdot)|_{\p M}$ or $d_F(\cdot,x)|_{\p M}$,
then the distance minimizing curve from $x$ to $z_x$ or from $z_x$ to
$x$ respectfully is a geodesic that is perpendicular to the
boundary. Readers who are not familiar with Finsler geometry are
encouraged to read this section before embarking to the proof of
Theorem~\ref{Th:smooth} presented in Section~\ref{Se:proof}.

Most of the claims and the proofs given in this section  are modifications of similar theorems in Riemannian geometry. We refer to the classical material where the Riemannian version is presented. 
\medskip

We let $N$ be a $n$-dimensional, compact, connected smooth manifold without boundary. We reserve the notation $TN$ for the tangent bundle of $N$ and say that a function $F\colon TN \to [0, \infty)$ is a \textit{Finsler function} if
\begin{enumerate}
\item $F\colon TN\setminus \{0\} \to [0, \infty)$ is smooth
\item For each $x \in N$ the restriction $F\colon T_xN \to [0, \infty)$ is a \textit{Minkowski norm}.
\end{enumerate}

Recall that for a vector space $V$ a function $F\colon V \to [0, \infty)$ is called a Minkowski norm if the following hold
\begin{itemize}
\item $F\colon V\setminus \{0\} \to \R $ is smooth.
\item For every $y \in V$ and $s  >0$ it holds that $F(sy)=sF(y)$.
\item For every $y \in V\setminus \{0\}$ the function $g_y\colon V \times V \to \R$ is a symmetric positive definite bilinear form, where
\begin{equation}
\label{eq:Finsler_to_Riemannian}
g_y(v,w):=\frac{1}{2}\frac{\p}{\p s}\frac{\p}{\p t}\bigg[F^2(y+sv+tw)\bigg]\bigg|_{s=t=0}.
\end{equation}
\end{itemize}
We call the pair $(N,F)$ a Finsler manifold.  

\medskip
The length of a piecewise smooth curve $c\colon I\to N$, $I$ is an interval, is defined as 
\begin{equation}
\label{eq:lenght_of_curve}
\mathcal{L}(c):=\int_I F(\dot{c}(t)) dt.
\end{equation}
For every $x_1,x_2 \in N$ we define
$$
d_F(x_1,x_2):=\inf_{c\in C_{x_1,x_2}} \mathcal{L}(c),
$$ 
where $C_{x_1,x_2}$ is the collection of piecewise smooth curves starting at $x_1$ and ending at $x_2$. The function $d_F\colon N \times N \to [0,\infty)$ is a non-symmetric path metric related to $F$, meaning that for some $x_1,x_2\in N$ the distance $d_F(x_1,x_2)$ need not coincide with $d_F(x_2,x_1)$ (see \cite[Section 6.2]{bao2012introduction}). 

We note that for all $x_1,x_2\in N$ it holds that
\begin{equation}
\label{eq:dF_and_dtildef}
d_F(x_1,x_2)=d_{\stackrel{\leftarrow}{F}}(x_2,x_1),
\end{equation}
where $\stackrel{\leftarrow}{F}$ is the reversed Finsler function $\stackrel{\leftarrow}{F}(x,y)=F(x,-y)$.

\medskip
We use the notation $g_{ij}(x,y)$ for the component functions of the Hessian of $\frac{1}{2}F^2$ as in~\eqref{eq:Finsler_to_Riemannian}.  A $C^1$ curve $\gamma\colon I \to N$, with a constant speed $F(\dot{\gamma}(t))\equiv c \geq 0$, is a geodesic of Finsler manifold $(N,F)$ if $\gamma(t)$ solves the system of\textit{ geodesic equations}
\begin{equation}
\label{eq:geodesic_eq}
\ddot \gamma^i(t) + 2G^i(\dot{\gamma}(t))=0, \quad i \in \{1,\ldots, m\}.  
\end{equation}
Here, $G^i\colon TN \to \R$ is given in local coordinates $(x,y)$ by
\begin{equation}
\label{eq:geodesic_coef_2}
G^i(x,y)=\frac{1}{4}g^{il}(x,y)\bigg\{2\frac{\p g_{jl}(x,y)}{\p x^k}-\frac{\p g_{jk}(x,y)}{\p x^l}\bigg\}y^jy^k.
\end{equation}
Since $F^2(x,y)$ is positively homogeneous of degree two with respect to $y$ variables, it follows from~\eqref{eq:geodesic_coef_2} that $G^i$ is positively
homogeneous of degree two with respect to $y$, but not necessarily
quadratic in $y$. Therefore, the geodesic equation~\eqref{eq:geodesic_eq} is not preserved if the orientation of the curve $\gamma$ is reversed. 

%
%
%

We define a vector field $\textbf{G}$, by
\begin{equation}
\label{eq:geodesic_vector_field}
\textbf{G}(x,y):= y^i\frac{\p}{\p x^i}-2G^i(x,y)\frac{\p}{\p y^i}.
\end{equation}
A curve  $\gamma$ is a geodesic of $F$ if and only if $\gamma=\pi(c)$, where $c$ is an integral curve of $\textbf G$. Due to ODE theory for a given initial conditions $(x,y)\in TN$ there exists the unique solution $\gamma_{x,y}$ of~\eqref{eq:geodesic_eq}, defined on maximal interval containing $0$. Thus by defining $\textbf{G}$ locally with~\eqref{eq:geodesic_vector_field}, it extends to a global vector field on $TN$. We call $\textbf{G}$ the\textit{ geodesic vector field}.

\begin{Lemma}[{\cite[Section 5.4]{shen2001lectures}}]
\label{Le:geodesic_flow_has_const_speed}
Let $c$ be an integral curve of geodesic vector field $\textbf G$, then $F(c(t))$ is a constant.
\end{Lemma}

We use the notations $\phi_t$ for the geodesic flow of $F$ on $TN$ and $(x,v)$ for points in $SN$. By Lemma~\ref{Le:geodesic_flow_has_const_speed} we know that for $(x,v)\in SN$ and for any $t\in \R$ in the flow domain of $(x,v)$ it holds that $\phi_{t}(x,v)\in SN$. Since 
$SN$ is compact we have proven that $\phi$ on $SN$ is a global flow (see for instance \cite[Theorem 17.11]{Lee}), which means that the map
\begin{equation}
\phi\colon \R\times SN \to SN
\end{equation}
is well defined. Therefore, we can define the exponential mapping $\exp_x$, $x \in N$ by
\begin{equation}
\label{eq:exponential_map}
\exp_x(y):=\pi(\phi_1(x,y))=\gamma_{x,y}(1), \quad  y \in T_xN.
\end{equation} 
Moreover in \cite[Section 11.4]{shen2001lectures}, it is shown that for any points $x_1,x_2 \in N$ there exists a globally minimizing geodesic from $x_1$ to $x_2$. 

In the following, we relate the smoothness of a distance function to distance minimizing property of geodesics. This is done via the cut distance function $\tau\colon {SN} \to \R$, which is defined by
\begin{equation}
\tau(x,v)=\sup \{t>0:d_F(x,{\gamma}_{x,v}(t))=t\}.
\end{equation}
In the next lemma, we collect properties of the cut distance function. 

\begin{Lemma}[{\cite[Chapter 12]{shen2001lectures}  or \cite[Chapter 8]{bao2012introduction}}]
\label{Le:cut_dist_func}
Let $(x, v) \in \; {SN}$ and $t_0 =\tau(x,v)$. 
At least one of the following holds:
\begin{enumerate}
\item The exponential map $\exp_x$, of $F$, is singular at $t_0v$.
\item There exists $\eta \in \;S_xN, \: \eta \neq v$ such that $\exp_x( t_0v)= {\exp_x}(t_0\eta)$.
\end{enumerate}
Moreover for any $t \in [0,t_0)$ the map ${\exp_x}$ is non-singular at $tv$. Also the map $\tau\colon {SN}\to \R$ is continuous.
\end{Lemma}


In the next lemma, we consider the regularity of the function $d_F$.

\begin{Lemma}
\label{Le:smoothness_of_the_distance_function}
Let $(x_1,v_1) \in SN$, $0<t_1 < \tau(x_1,v_1)$ and $x_2=\gamma_{x_1,v_1}(t_1)$. Then there exists neighborhoods $U$ of $x_1$ and $V$ of $x_2$ respectively such that the distance function $d_F\colon U\times V \to \R$ is smooth.
\end{Lemma}

\begin{proof}
Since the cut distance function $\tau$ is continuous, there exist a neighborhood $U'\subset SM$ of $(x_1,v_1)$ and $\epsilon>0$ such that for any $t \in (t_1-\epsilon,t_1+\epsilon)$ and $(x,v) \in U'$ holds $t<\tau(x,v)$.  

Consider the smooth function 
\begin{equation}
E\colon U' \times (t_1-\epsilon,t_1+\epsilon) \ni ((x,v),t) \mapsto (x,\exp_xtv) \in N \times N.
\end{equation}
Since for every $ ((x,v),t) \in U' \times (t_1-\epsilon,t_1+\epsilon)$ we have that the exponential map $\exp_x$ is not singular at $vt \in T_xN$, the Jacobian of $E$ is invertible in $U' \times (t_1-\epsilon,t_1+\epsilon)$. Thus the Inverse Function Theorem implies the existence of the neighborhood $U\times V \subset N\times N$ of $(x_1,x_2)$ such that $E$ is a diffeomorphism onto $U \times V$. Therefore the map
\begin{equation}
U \times V \ni (x,y) \mapsto \exp^{-1}_xy \in TN
\end{equation}
is smooth.

By the definition of the cut distance function and  \cite[Section 11.4]{shen2001lectures}, the following equation holds for any $(x,y) \in U \times V$,
\begin{equation}
d_F(x,y)=F(x,\exp_x^{-1}y).
\end{equation}
This implies the claim as $F$ is smooth outside the zero section.
\end{proof}

The duality map between the tangent bundle and the cotangent bundle is given by the \textit{Legendre transform} $\ell\colon TN\setminus\{0\}\to T^\ast N \setminus\{0\}$ which is defined by
\begin{equation}
\label{eq:Legendre}
\ell(x,y)=\ell_x(y):=g_y(y,\cdot) \in T^\ast_xN, \quad y \in T_xN.
\end{equation}
The Legendre transform is a diffeomorphism and for all $a>0$ and $(x,y)\in TN\setminus \{0\}$ we have
\begin{equation}
\label{eq>legendre_scale}
\ell(x,ay)=a\ell(x,y).
\end{equation} 
(see \cite[Section 3.1]{shen2001lectures}). The dual $F^\ast$ of the Finsler function $F$, which is given by
\begin{equation}
\label{eq:F_dual}
F^\ast(x,p):=\sup_{v\in S_xN}p(v), \quad (x,p) \in T^\ast  N,
\end{equation}
is a Finsler function on $T^\ast N$ and the Legendre transform $\ell_x$ satisfies
\begin{equation}
F(x,v)=F^\ast(x,\ell_x(v)).
\end{equation}

\medskip

We let $S \subset N$ be a smooth submanifold of co-dimension $1$. It is shown in \cite[ Section 2.3]{shen2001lectures} that for every $z \in S$ there exists precisely two unit vectors $\nu_1,\nu_2 \in S_zN$ such that
$$
T_zS=\{y \in T_zN: g_{\nu_i}(\nu_i,y)=0\}, \: i\in \{1,2\}.
$$
Vectors $\nu_1,\nu_2 \in S_pN$ are called the unit normals of $S$. Typically $\nu_1\neq -\nu_2$ unless~$F$ is reversible.

In the next lemma, we relate the Legendre transform of the velocity field of a distance minimizing geodesic to the differential of the distance function.

\begin{Lemma}
\label{Le:differential_of_dist_func}
Let $x_1 \in N$ and $x_2 \in N$ be such that $d_F(x_1,\cdot)$ is smooth at $x_2$.
Then 
\begin{equation}
\label{eq:differential_of_dist_func}
d(d_F(x_1,\cdot))\bigg|_{x_2}=g_{\dot{\gamma}_{x_1,v}(t)}(\dot{\gamma}_{x_1,v}(t),\cdot)\bigg|_{t=d_F(x_1,x_2)}\in T^\ast_{x_2}N,
\end{equation}
where $\gamma_{x_1,v}$ is the unique distance minimizing unit speed geodesic from $x_1$ to $x_2$.
\end{Lemma}
\begin{proof}
Denote $t_0=d_F(x_1,x_2)$ and 
\begin{equation}
S(x_1,t_0)=\exp_{x_1}\{w \in T_{x_1}N: F(w)=t_0\}.
\end{equation}
Recall that 
\begin{equation}
\label{eq:d_F=F}
d_F(x_1,\exp_{x_1}(tw))=F(tw)=t, \quad  t>0, \: w\in S_{x_1}N
\end{equation}
if $tw$ is close to $t_0 v$. We  use a shorthand notation $d_F$ for the function $d_F(x_1,\cdot)$. We take a $t$-derivative from the both sides of~\eqref{eq:d_F=F} to obtain
\begin{equation}
\label{eq:d_F=F_2}
d(d_F)\bigg|_{\exp_{x_1}(tw)}(D\exp_{x_1}|_{tw}w)=d(d_F)\bigg|_{\exp_{x_1}(tw)}(\dot{\gamma}_{x_1,w}(t))=1.
\end{equation}
Due to~\eqref{eq:d_F=F_2} the set $S(x_1,t_0)$ is a regular level set of $d_F$ near $x_2$, and moreover~\eqref{eq:d_F=F} implies
\begin{equation}
T_{x_2}S(x_1,t_0)=\ker d(d_F)\bigg|_{x_2}.
\end{equation}

Thus it suffices to prove that 
\begin{equation}
\ker g_{\dot{\gamma}_{x_1,v}(t_0)}(\dot{\gamma}_{x_1,v}(t_0),\cdot)=T_{x_2}S(x_1,t_0).
\end{equation}
Notice that for any $w \in T_{x_1}M,$ such that  $g_v(v,w)=0$ holds 
\begin{equation}
\label{eq:d_F=F_3}
0=g_{tv}(tv,tw)=\frac{1}{2}\frac{d}{ds}[F^2](t(v+sw))\bigg|_{s=0}=t_0 d(d_F)\bigg|_{\exp_{x_1}(tv)}(D\exp_{x_1}|_{tv}tw).
\end{equation}
Therefore, $d(d_F)|_{\exp_{x_1}(tv)}(D\exp_{x_1}|_{tv}tw)=0$ and $(D\exp_{x_1}|_{t_0v}t_0w) \in T_{x_2}S({x_1},t_0)$.  
Recall that $J(t):= D\exp_{x_1}|_{tv}tw$ is the unique Jacobi field with initial conditions $J(0)=0, \: D_tJ(0)=w$. 
By Gauss' Lemma \cite[ Lemma 11.2.1]{shen2001lectures} we have 
\begin{equation}
0=g_{v}(v,w)=g_{\dot{\gamma}_{{x_1},v}(t_0)}(\dot{\gamma}_{{x_1},v}(t_0),D\exp_{x_1}|_{t_0v}w)).
\end{equation}
In the above, we used the identity
\begin{equation}
g_{tv_1}(tv_1,tv_2)=t^2g_{v_1}(v_1,v_2), \quad t>0;\: v_1,v_2 \in T_xN.
\end{equation}
This implies that
\begin{equation}
\ker g_{\dot{\gamma}_{{x_1},v}(t_0)}(\dot{\gamma}_{{x_1},v}(t_0),\cdot)=\{D\exp_{x_1}|_{t_0v}w: g_v(v,w)=0\}= T_{x_2}S({x_1},t_0),
\end{equation}
since $\dim v^\perp =\dim T_{x_2}S({x_1},t_0)$ and $D\exp_{x_1}|_{t_0v}$ is not degenerate.
\end{proof}

\begin{Lemma}
\label{Le:normal_geo_is_mini}
Let $S \subset N$ be a smooth closed submanifold of co-dimension $1$. Let $x \in N$. A distance minimizing curve from $S$ to $x$ (from $x$ to $S$) is a geodesic  that is orthogonal to $S$ at the initial (terminal) point. 
%
%
\end{Lemma}

\begin{proof}
Since $S$ is compact there exists a closest point $z_x\in S$ to $x$. We denote $h=d_F(x,z_x)$. Since $(N,F)$ is complete there exists a distance minimizing geodesic $\gamma$ from $x$ to~$z_x$.

We suppose first that $d_F(x,\cdot)$ is smooth at $z_x$. We denote $r(z)=d_F(x,z)$ for $z \in S$. Since $z_x$ is a minimal point of $r$ we have $d_S r(z_x)=0$. Here $d_S$ is the differential operator of smooth manifold $S$. Then $d_Sr=\iota ^\ast d(d_F(x,\cdot))$, where $\iota: S \hookrightarrow N$. Thus $d (d_F(x,\cdot))$ vanishes on $T_{z_x}S$. By~\eqref{eq:differential_of_dist_func} it holds that 
\begin{equation}
d(d_F(x,\cdot))|_{z_x}=g_{\dot \gamma(h)}(\dot \gamma(h),\cdot)\neq 0.
\end{equation}
Thus $\dot \gamma(h)$ is normal to $S$ at $z_x$. 


If $d_F(x,\cdot)$ is not smooth at $z_x$ there exists $\epsilon>0$ such that for any $t\in (\epsilon,h)$ $d_F(\gamma(t),\cdot)$ is smooth at $z_x$. By the first part of the proof it follows that $\dot \gamma(h)$ is perpendicular to~$S$. 

\medskip

Due to~\eqref{eq:dF_and_dtildef} the second claim for the reversed
distance function can be proven in the same way, upon replacing $F$ by~$\stackrel{\leftarrow}{F}$.
\end{proof}

\bibliographystyle{abbrv}
\bibliography{bibliography}

\def\cprime{$'$}

\end{document}